\documentclass[preprint,12pt]{elsarticle}

\usepackage{amssymb}
\usepackage[english]{babel}
\usepackage[letterpaper,top=2cm,bottom=2cm,left=2cm,right=2cm,marginparwidth=1.75cm]{geometry}
\usepackage{amssymb}
\usepackage{amsmath}
\usepackage{amsthm}
\usepackage{mathrsfs}
\usepackage{caption}
\usepackage{subcaption}
\usepackage{booktabs}
\usepackage{array,caption, threeparttable}
\usepackage[font=small,labelfont=bf,labelsep=period]{caption}
\usepackage[colorlinks,linkcolor=red,anchorcolor=black,citecolor=green]
{hyperref}
\usepackage{cleveref}
\usepackage{prettyref}
\usepackage{hyperref}
\usepackage{graphicx}
\usepackage{color}
\usepackage{CJK}
\usepackage{soul}
\usepackage{indentfirst}
\usepackage{multirow}
\usepackage{algorithm}
\usepackage{algorithmic}
\usepackage{tikz}
\usetikzlibrary{calc}
\makeatletter
\newcommand{\rmnum}[1]{\romannumeral #1}
\newcommand{\Rmnum}[1]{\expandafter\@slowromancap\romannumeral #1@}
\makeatother
\newtheorem{lemma}{Lemma}[section]
\newtheorem{theorem}{Theorem}[section]
\newtheorem{definition}{Definition}[section]
\numberwithin{equation}{section}

\journal{ }

\begin{document}

\begin{frontmatter}



\title{A tensor SVD-like decomposition based on the semi-tensor product of tensors}


\author[add1]{Zhuo-Ran Chen}
\ead{20zrchen@stu.edu.cn}
\author[add2]{Seak-Weng Vong}
\ead{swvong@um.edu.mo}
\author[add1]{Ze-Jia Xie\corref{cor1}}
\ead{zjxie@stu.edu.cn}
\cortext[cor1]{Corresponding author}
\address[add1]{Department of Mathematics, College of Science, Shantou University, Shantou, 515063, China.}
\address[add2]{Department of Mathematics, Faculty of Science and Technology, University of Macau, Macao, China.}

\tnotetext[label1]{The second author was supported by the research grant MYRG2020-00035-FST from University of Macau. The third author was supported by the National Natural Science Foundation of China (No. 11801074), key research projects of general universities in Guangdong Province (No. 2019KZDXM034), basic research and applied basic research projects in Guangdong Province (Projects of Guangdong, Hong Kong and Macao Center for Applied Mathematics) (No. 2020B1515310018), and the Shantou University Start-up Funds for Scientific Research (No. NTF19035).}
\begin{abstract}
\indent In this paper, we define a semi-tensor product for third-order tensors. Based on this definition, we present a new type of tensor decomposition strategy and give the specific algorithm. This decomposition strategy actually generalizes the tensor SVD based on semi-tensor product. Due to the characteristic of semi-tensor product for compressing the data scale, we can therefore achieve data compression in this way. Numerical comparisons are given to show the advantages of this decomposition.
\end{abstract}




\end{frontmatter}


\section{Introduction}

Nowadays, many kinds of fields need to collect and apply a large amount of data, such as image and video processing, medical treatment and engineering. In many cases, data are multidimensional, such as the storage of color pictures, video clips, etc. However, two-dimensional matrices are not enough when analyzing and processing these data, tensors are therefore introduced to analyze multidimensional data. From that, the storage and decomposition of tensor is a very important research content. In practical application, we often need to store and process the huge number of data, so how to reduce the storage space and improve operation speed is also a very important research area for us.\\
\indent CANDECOMP/PARAFAC (CP) \cite{RePEc:spr:psycho:v:35:y:1970:i:3:p:283-319, harshman1970foundations} and Tucker \cite{tucker1966some} decompositions are two well-known tensor decomposition strategies. They are higher-order extensions of the matrix singular value decomposition (SVD). The CP model can decompose the target tensor into the sum of rank-one tensors. The Tucker decomposition applies the n-mode multiplication of tensors to decompose a target tensor into a core tensor multiplied by some matrices along their modes. In fact, many other tensor decompositions \cite{kolda2009tensor} have been developed, including INDSCAL \cite{RePEc:spr:psycho:v:35:y:1970:i:3:p:283-319}, PARAFAC2 \cite{harshmannote}, CANDELINC \cite{douglas1980candelinc}, DEDICOM \cite{harshman1978models}, PARATUCK2 \cite{harshman1996uniqueness}, and so on. 
In 2008, Kilmer, Martin, and Perrone \cite{t-product} presented a so-called t-product and developed a new decompositon based on this t-product. Such decomposition can write a third-order tensor as products of three other third-order tensors. In order to define this new notion, they give a new definition of tensor-tensor multiplication firstly. This decomposition called T-SVD, is analogous to the SVD in the matrix case, and it could be used in field of data compression.\\ 
\indent In this paper, we define a new tensor multiplication for third-order tensors which extends the matrix semi-tensor product to tensors. Besides, we introduce a new tensor decomposition algorithm based on the semi-tensor product which can decompose the target tensor into three other tensors and multiply. We first give the matrix SVD based on semi-tensor product, and then generalize it to tensor. This decomposition will reduce the storage of data and improve the speed of operation to a certain extent, especially for the tensors with huge amounts of data. \\
\indent Our paper is organized as follows. In Section 2, we introduce the notation and preliminaries that we need to use throughout the paper. In Section 3, we describe the new semi-tensor product that we define in detail, including some useful properties, lemmas, and theorems. In Section 4, we will use the tensor product introduced in Section 3 to decompose a third-order tensor so that the target third-order tensor can be written in the form of three tensor products and multiply. In Section 5, we give the application in image processing by using the algorithm introduced in this paper. And conclusions are made in Section 6.

\section{Notations and Preliminaries}\label{sec:2} 
In this section, we will give a summary of notations and basic preliminaries that we will use. The arrangement of data in the same direction is called a one-way array. A tensor is a multi-way array representation of data, it is a multi-way array or a multi-dimensional array, which is an extension of a matrix. The tensor represented by a $p$-way array is a tensor of order $p$. An order-$p$ tensor $\mathcal{A}$ can be written as $\mathcal{A}= \left( {a_{i_{1}i_{2} \dots i_{p}}} \right)  \in \mathbb{C}^{n_{1}\times n_{2} \times  \dots \times n_{p}}$ \cite{kiers2000towards}, where $a_{i_{1}i_{2} \dots i_{p}}$ denotes the $\left( {i_{1}i_{2} \dots i_{p}}\right) $-th entry of  $\mathcal{A}$. Thus, a matrix can be considered as a second-order tensor, and a vector is a first-order tensor. \\ 
\indent For a matrix $A \in \mathbb{C}^{n_{1} \times n_{2}}$, we use $A^{T}$ and $A^{H}$ to denote the transpose and conjugate transpose of $A$, respectively. For a third-order tensor $\mathcal{A}=(a_{ijk}) \in \mathbb{C}^{n_{1}\times n_{2}\times n_{3}}$, we can use the set of matrices to denote its horizontal, lateral, and frontal slices. The $i$-th horizontal slice, $j$-th lateral slice, and $k$-th frontal slice of the third-order tensor $\mathcal{A} $ can be denoted by $\mathcal{A}\left( i,:,:\right) $,  $\mathcal{A}\left( :,j,:\right) $, and  $\mathcal{A}\left( :,:,k\right) $, respectively. Each slice is a matrix actually. The slice graph of a third-order tensor is shown in Figure \ref{figure1}.\\
\begin{figure}[h] 
	\begin{center}
		\begin{tikzpicture}[scale=1.1]
		\node[scale=1.3] at (0,0.5) {
			\begin{tikzpicture}
			\draw (0,0)coordinate(b) -- ++(1,0)coordinate(c) --++(50:0.5)coordinate(a) -- ++(-1,0)coordinate(d) -- cycle;
			\foreach \x in{a,b,c}
			\draw (\x) -- ++(0,-1);
			\draw (0,-1) -- ++(1,0) -- ++(50:0.5);
			\node[scale=0.35] at (-0.1,-0.5) {\rotatebox{90}{$i=1,2, \dots, n_{1}$}};
			\node[scale=0.35] at (0.45,-1.15) {$j=1,2, \dots, n_{2}$};
			\node[scale=0.35] at (1.33,-0.75) {\rotatebox{55}{$k=1,2, \dots, n_{3}$}};
			\node[scale=0.5] at (0.7,-1.7) {$\mathcal{A}=(a_{ijk})\in \mathbb{C}^{n_{1}\times n_{2}\times n_{3}}$};
			\end{tikzpicture}};
		
		\node[scale=2] at (-4,-4) {
			\begin{tikzpicture}
			\foreach \x in {-1,-0.2,0}
			\draw[fill=white] (0,\x) -- ++(0.75,0) -- ++(60:0.35) -- ++(-0.75,0) -- cycle;
			\node[scale=0.3] at (0.4,0.15) {$\mathcal{A}(1,:,:)$};
			\node[scale=0.3] at (0.4,-0.1) {$\mathcal{A}(2,:,:)$};
			\node[scale=0.6] at (0.4,-0.4) {$\vdots$};
			\node[scale=0.3] at (0.4,-0.85) {$\mathcal{A}(n_{1},:,:)$};
			\node[scale=0.5] at (0.4,-1.5) {Horizontal slices};  
			\end{tikzpicture}};
		
		\node[scale=1.5] at (0,-4.2) {
			\begin{tikzpicture}
			\draw[fill=white] (-0,0) -- ++(30:0.65) -- ++(90:0.8) -- ++(210:0.65) -- cycle;
			\draw[fill=white] (0.4,0.1)  -- ++(30:0.65) -- ++(90:0.8) -- ++(210:0.65) -- cycle;
			\draw[fill=white] (1.5,0.13)  -- ++(30:0.65) -- ++(90:0.8) -- ++(210:0.65) -- cycle;
			\node[scale=0.35,xscale=0.8] at (0.2,0.65) {$\mathcal{A}(:,1,:)$};
			\node[scale=0.35,xscale=0.9] at (0.7,0.65) {$\mathcal{A}(:,2,:)$};
			\node[scale=0.35,xscale=0.9] at (1.79,0.65) {$\mathcal{A}(:,n_{2},:)$};
			\node[scale=0.7] at (1.25,0.65) {$\dots$};
			\node[scale=0.65] at (1,-0.85) {Lateral slices};  
			\end{tikzpicture}  
		};
		
		\node[scale=2.4] at (4,-4){
			\begin{tikzpicture}
			\draw[fill = white] (0.6,0.6) rectangle ++(0.5,0.5);
			\draw[fill = white] (0.2,0.2) rectangle ++(0.5,0.5);
			\draw[fill = white] (0,0) rectangle ++(0.5,0.5);
			\node[scale=0.3] at (0.25,0.25) {$\mathcal{A}(:,:,1)$};
			\node[scale=0.3] at (0.45,0.57) {$\mathcal{A}(:,:,2)$};
			\node[scale=0.3] at (0.85,0.85) {$\mathcal{A}(:,:,n_{3})$};
			\node[rotate=45,scale=0.5] at (0.8,0.5) {$\dots$};
			\node[scale=0.4] at (0.55,-0.4) {Frontal slices};    
			\end{tikzpicture}
		};
		
		\node[scale=1.2] at (0,-2) {$\Downarrow$};
	\end{tikzpicture}
\end{center}
\caption{Slice maps of an $n_{1}\times n_{2}\times n_{3}$ third-order tensor $\mathcal{A}$.}
\label{figure1}
\end{figure}
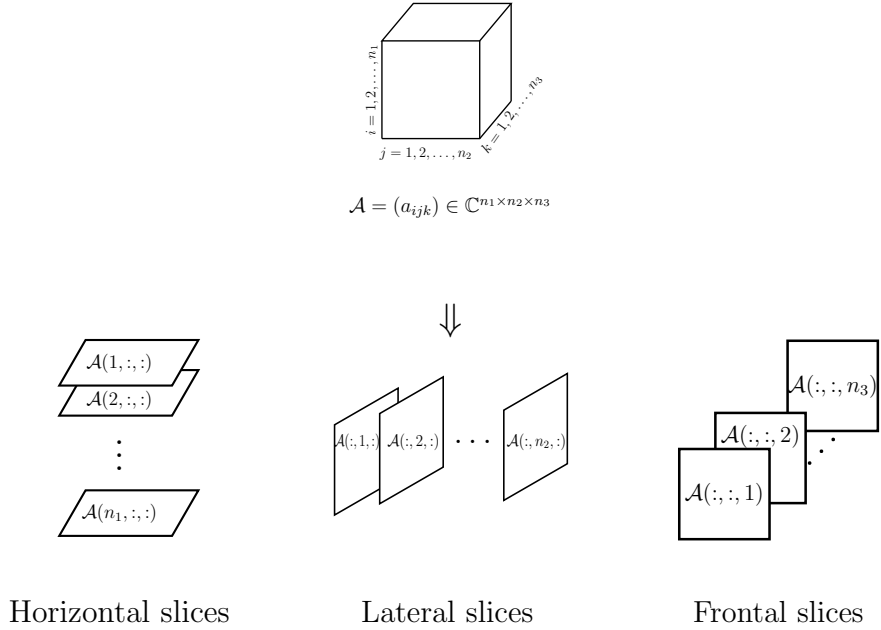
\indent In fact, if $\mathcal{A} \in \mathbb{C}^{n_{1}\times n_{2}\times n_{3}}$ is a third-order tensor, then it has $n_{3}$ frontal slices, and each frontal slice is an $n_{1}\times n_{2}$ matrix. These $n_{3}$ frontal slices of $\mathcal{A}$ can be represented by $\mathit{A}_{1}=\mathcal{A}\left( :,:,1\right)$, $\mathit{A}_{2}=\mathcal{A}\left( :,:,2\right)$, $\dots$,  $\mathit{A}_{n_{3}}=\mathcal{A}\left( :,:,n_{3}\right)$  \cite{kolda2006multilinear}. We give a $2\times2\times2$ example in Figure \ref{figure2}.\\	
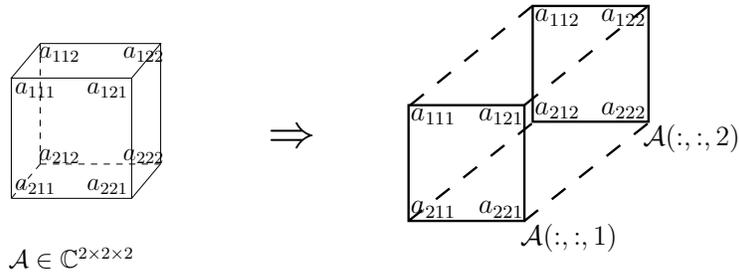
\begin{figure}[hb] 
	\begin{center}
		\begin{tikzpicture}[scale=1.3]
		\node[scale=0.8] at (6,0.5) {
			\begin{tikzpicture}
			\draw (0,0)coordinate(b) -- ++(2,0)coordinate(c) --++(50:0.75)coordinate(a) -- ++(-2,0)coordinate(d) -- cycle;
			\foreach \x in{a,b,c}
			\draw (\x) -- ++(0,-2);
			\draw (0,-2) -- ++(2,0) -- ++(50:0.75);
			\node[] at (0.4,-0.2) {$a_{111}$};
			\node[] at (1.6,-0.2) {$a_{121}$};
			\node[] at (0.4,-1.8) {$a_{211}$};
			\node[] at (1.6,-1.8) {$a_{221}$};
			\node[] at (0.8,0.4) {$a_{112}$};
			\node[] at (2.2,0.4) {$a_{122}$};
			\node[] at (0.8,-1.3) {$a_{212}$};
			\node[] at (2.2,-1.3) {$a_{222}$};
			\draw[dashed] (0.48,0.57)  -- (0.48,-1.43);
			\draw[dashed] (0.48,-1.43)  -- (2.47,-1.43);
			\draw[dashed]  (0.48,-1.43) -- (0,-2);
			\node[] at (1,-3) {$\mathcal{A}\in \mathbb{C}^{2\times2\times2}$};
			\end{tikzpicture}};
		\node[scale=2.2] at (11,0.8){
			\begin{tikzpicture}			
			\draw[fill = white] (0.75,0.6) rectangle ++(0.7,0.7);
			\draw[fill = white] (0,0) rectangle ++(0.7,0.7);
			\node[scale=0.4] at (0.15,0.63) {$a_{111}$};
			\node[scale=0.4] at (0.56,0.63) {$a_{121}$};
			\node[scale=0.4] at (0.15,0.07) {$a_{211}$};
			\node[scale=0.4] at (0.56,0.07) {$a_{221}$};
			\node[scale=0.4] at (0.9,1.23) {$a_{112}$};
			\node[scale=0.4] at (1.3,1.23) {$a_{122}$};	
			\node[scale=0.4] at (0.9,0.67) {$a_{212}$};			
			\node[scale=0.4] at (1.3,0.67) {$a_{222}$};	
			\draw[dashed] (0.01,0.7) -- (0.81,1.35);	
			\draw[dashed] (0.7,0.7) -- (1.51,1.35);	
			\draw[dashed] (0.7,0) -- (1.51,0.64);	
			\draw[dashed] (0,0) -- (0.81,0.64);	
			\node[scale=0.4] at (0.97,-0.1) {$\mathcal{A}(:,:,1)$};
			\node[scale=0.4] at (1.71,0.5) {$\mathcal{A}(:,:,2)$};
			\end{tikzpicture}
		};			
		\node[scale=1.5] at (8.1,0.7) {$\Rightarrow$};
	\end{tikzpicture}
\end{center}
\caption{The frontal slices of $\mathcal{A} \in \mathbb{C}^{2\times2\times2} $.}
\label{figure2}
\end{figure}
\indent	The semi-tensor product of matrices was proposed by Cheng \cite{cheng2007survey}. As we know, the multiplication of two matrices requires the strict matching condition. While the semi-tensor product release the requirement for dimensions, it only requires dimensions to be multiples. See Definition \ref{definition 2.2} for more detail. The semi-tensor product includes left semi-tensor product and right semi-tensor product. We only discuss the left semi-tensor product in this paper.
\begin{definition}  \label{definition 2.1}
	(Left Semi-tensor Product of Vectors \cite{cheng2007survey}) Let $\mathbf{x}=\left[\mathit{x_{1}, x_{2}, \dots, x_{p}}\right]^{T}$ $\in \mathbb{C}^{p} $ and $\mathbf{y}=\left[\mathit{y_{1}, y_{2}, \dots, y_{q}}\right]^{T}\in \mathbb{C}^{q} $. Then we define left semi-tensor product of two vectors  $\mathbf{x}^{T}$ and $\mathbf{y}$ in the following.\\
	$(\rmnum{1}) $  If $p=nq, n\in \mathbb{Z}^{+}$, then we divide the row vector  $\mathbf{x}^{T}$ into  $q$ blocks: $\mathbf{x}_{1}^{T}, \mathbf{x}_{2}^{T}, \dots, \mathbf{x}_{q}^{T}$. Each block $\mathbf{x}_{i}^{T}$ is an $n$-dimensional row vector, and is multiplied by ${y}_{i}$, respectively, and then add all. We can get the left semi-tensor product of $\mathbf{x}^{T}$ and $\mathbf{y}$ which is an $n$-dimensional row vector. It can be represented as \\
	\begin{equation*}
	\mathbf{x}^{T} \ltimes \mathbf{y} = \sum_{i=1}^{q}\mathbf{x}_{i}^{T}{y}_{i}\in \mathbb{C}^{1\times n}.
	\end{equation*}		 
	$(\rmnum{2})$  If $p=\dfrac{1}{n}q, n\in \mathbb{Z}^{+}$, then we divide the column vector  $\mathbf{y}$ into  $p$ blocks: $\mathbf{y}_{1}, \mathbf{y}_{2}, \dots, \mathbf{y}_{p}$. Each block $\mathbf{y}_{i}$ is an $n$-dimensional column vector, and is multiplied by ${x}_{i}$, respectively, and then add all. We can get the left semi-tensor product of $\mathbf{x}^{T}$ and $\mathbf{y}$ which is an $n$-dimensional column vector. It can be represented as \\
	\begin{equation*}
	\mathbf{x}^{T} \ltimes \mathbf{y} = \sum_{i=1}^{p}{x}_{i}\mathbf{y}_{i}\in \mathbb{C}^{n}.
	\end{equation*}
\end{definition}
\begin{definition} \label{definition 2.2}
	(Left Semi-tensor Product of Matrices \cite{cheng2007survey})
	Let  $A\in \mathbb{C}^{m\times n}$, $B\in \mathbb{C}^{s\times t}$. If $n=ks$, or $ n=\dfrac{1}{k}s$, $ k\in \mathbb{Z}^{+} $, then $D=A\ltimes B$ is a block matrix which has $m\times t$ blocks, called the left semi-tensor product of $A$ and $B$. Each block can be represented as  
	\begin{equation*}
	D^{ij}={A}^{i}\ltimes {B}_{j},\quad i=1,2, \dots, m, j=1,2, \dots, t,
	\end{equation*}
	where ${A}^{i}$ is the i-th row of $A$, ${B}_{j}$ is the j-th column of $B$.
\end{definition}
\indent  Only when  $n$ and $s$ are integer multiples, Definition \ref{definition 2.2} is well-defined. If $n=ks$, $D=A\ltimes B$ is an $m\times kt$ matrix; If $ n=\dfrac{1}{k}s$, $D=A\ltimes B$ is a $km\times t$ matrix; If $n=s$, the left semi-product becomes the general matrix product.
\begin{definition} \label{definition 2.3}
	(Kronecker Product of Matrices \cite{kroneckerproduct})
	Suppose $A$ is an $m\times n$ matrix, and $B$ is an $s\times t$ natrix. Then $C=A\otimes B$ is an $ms \times nt$ matrix, called the Kronecker product of $A$ and $B$. It is defined as
	\begin{equation*}
	A\otimes B=
	\left[
	\begin{matrix}
	a_{11}B & a_{12}B & \cdots & a_{1n}B \\
	a_{21}B & a_{22}B & \cdots & a_{2n}B \\
	\vdots & \vdots& \ddots & \vdots  \\
	a_{m1}B & a_{m2}B & \cdots & a_{mn}B 
	\end{matrix}
	\right]. 
	\end{equation*}
\end{definition}

We give some basic properties of the matrix Kronecker product in the following.
\begin{lemma} (\cite{kroneckerproduct, van2000ubiquitous})\label{lemma 2.1}
	Suppose $A,B,C,D$ are four matrices of proper dimensions, and $\alpha$ is a scalar. We have \\ 
	$(\rmnum{1})$  $AB\otimes CD=(A\otimes C)(B\otimes D);$\\
	$(\rmnum{2})$  $A \otimes\left( B\pm C\right) =(A\otimes B)\pm(A\otimes C)$, $\left( B\pm C\right) \otimes A = B \otimes A \pm C \otimes A;$\\
	$(\rmnum{3})$  $ (A \otimes B)^{T}=A^{T} \otimes B^{T}$, $ (A \otimes B)^{H}=A^{H} \otimes B^{H};$\\
	$(\rmnum{4})$  $ (A \otimes B)^{-1}=A^{-1} \otimes B^{-1}$ if $A$ and $B$ are invertible;\\
	$(\rmnum{5})$  $\left( A \otimes B\right)\otimes C =A\otimes (B\otimes C);$ \\ 
	$(\rmnum{6})$  $\left(\alpha A \right) \otimes B  =A \otimes \left(\alpha B\right)=\alpha \left( A \otimes B\right).$  	
\end{lemma}
\indent We use  $I_{n}$ to denote the $n \times n$ identity matirx, then in Definition \ref{definition 2.2}, the semi-tensor product can be represented by Kronecker product as follows.    
\begin{lemma} (\cite{cheng2012introduction}) \label{lemma 2.2}
	Let $A$ be an $m \times n$ matrix, and $B$ be an $s \times t$ matrix. Suppose that they have proper dimensions such that the semi-tensor product is well-defined. Then we have \\
	$(\rmnum{1})$ If $n=ks$, $ k\in \mathbb{Z}^{+} $, then $A\ltimes B=A(B\otimes I_{k} )$ is an $m\times kt$ matrix;\\
	$(\rmnum{2})$If $ n=\dfrac{1}{k}s$, $ k\in \mathbb{Z}^{+} $, then $A\ltimes B=(A\otimes I_{k} )B$ is a $km\times t$ matrix.
\end{lemma} 
\begin{lemma} (\cite{cheng2012introduction}) \label{lemma 2.3}
	Suppose $A$, $B$, $C$ are matrices which have proper dimensions such that the $\ltimes$ is well-defined. Then we have\\
	\begin{equation*}
	(A\ltimes B)\ltimes C = A \ltimes (B\ltimes C).
	\end{equation*}     
\end{lemma}     
\noindent Lemma \ref{lemma 2.3} is called the associative law of semi-tensor product of matrices.\\ 
\indent  We know that if $\mathbf{\mathbf{v}}=\left[\mathit{v_{0}},\mathit{v_{1}},\mathit{v_{2}},\mathit{v_{3}} \right] ^{T}$ is a column vector, then \\
\begin{equation*}
\mathrm{circ}\left( \mathbf{v}\right) = 
\left[
\begin{matrix}
\mathit{v_{0}} &\mathit{v_{3}} & \mathit{v_{2}} &\mathit{v_{1}} \\
\mathit{v_{1}} & \mathit{v_{0}} & \mathit{v_{3}} &\mathit{v_{2}} \\
\mathit{v_{2}} & \mathit{v_{1}}& \mathit{v_{0}} & \mathit{v_{3}}  \\
\mathit{v_{3}} & \mathit{v_{2}}& \mathit{v_{1}} & \mathit{v_{0}} 
\end{matrix}
\right] 
\end{equation*}
is a circulant matrix. Note that the matrix is determined by the first column of the vector $ \mathbf{v}$.\\
\indent  Suppose $\mathbf{\mathbf{v}}$  is an $n \times 1$ column vector, then $\mathrm{circ}\left( \mathbf{v}\right)$ is an $n \times n$ circulant matrix, we can use normalized Fourier transform \cite{golub2013matrix} to change this circulant matrix into a diagonal matrix. That is multiplying $n \times n$ Fourier matrix on the left and right sides of the circulant matrix, respectively. If $F_{n}$ is an $n \times n$ Fourier matrix, then $F_{n}\mathrm{circ}\left( \mathbf{v}\right) {F_{n}^{H}}$ is a diagonal matrix, and the diagonal of $F_{n} \mathrm{circ}\left( \mathbf{v}\right) {F_{n}^{H }}$ is the Fourier transform result of the vector $\mathbf{\mathbf{v}}$. Besides, the Fourier transform also can convert a block-circulant matrix into a block-diagonal matrix.\\   
\indent In \cite{t-product}, the $\mathbf{fold}$ and $\mathbf{unfold}$ operators are defined. Suppose $\mathcal{A} \in \mathbb{C}^{n_{1} \times n_{2} \times n_{3}}$ is a third-order tensor, and its frontal slices are denoted by $\mathcal{A}\left( :,:,1\right)$, $\mathcal{A}\left( :,:,2\right)$, $\dots$, $\mathcal{A}\left( :,:,n_{3}\right)$. Then $\mathbf{unfold}$ is defined as
\begin{equation*}
\mathrm{unfold}\left( \mathcal{A}\right)=\left[
\begin{matrix}
\mathcal{A}\left( :,:,1\right)  \\
\mathcal{A}\left( :,:,2\right)  \\
\vdots   \\
\mathcal{A}\left( :,:,n_{3}\right)  
\end{matrix}
\right] \in \mathbb{C}^{n_{1}n_{3} \times n_{2}}.
\end{equation*} 
We can easily see that $ \mathrm{unfold}\left( \mathcal{A}\right)$ is an ${n_{1}n_{3} \times n_{2}}$ matrix. And its schematic is shown in Figure \ref{figure3}. Another operator $\mathbf{fold}$ is the inverse of $\mathbf{unfold}$ which is defined as
\begin{equation*}
\mathrm{fold}\left(\mathrm{unfold}( \mathcal{A})\right)=\mathrm{fold}\left( \left[
\begin{matrix}
\mathcal{A}\left( :,:,1\right)  \\
\mathcal{A}\left( :,:,2\right)  \\
\vdots   \\
\mathcal{A}\left( :,:,n_{3}\right)  
\end{matrix}
\right] \right)=\mathcal{A}.
\end{equation*} 
\begin{figure}[h] 
	\begin{center}
		\begin{tikzpicture}[scale=0.8]
		\node[scale=1.6] at (-1.5,0.5) {
			\begin{tikzpicture}
			\draw (0,0)coordinate(b) -- ++(1,0)coordinate(c) --++(50:0.55)coordinate(a) -- ++(-1,0)coordinate(d) -- cycle;
			\foreach \x in{a,b,c}
			\draw (\x) -- ++(0,-1);
			\draw (0,-1) -- ++(1,0) -- ++(50:0.55);
			\node[scale=0.35] at (-0.1,-0.5) {\rotatebox{90}{$i=1,2, \dots, n_1$}};
			\node[scale=0.35] at (0.45,-1.15) {$j=1,2, \dots, n_2$};
			\node[scale=0.35] at (1.35,-0.75) {\rotatebox{55}{$k=1,2, \dots, n_3$}};
			\node[scale=0.55] at (0.7,-1.7) {$\mathcal{A}=(a_{ijk})\in \mathbb{C}^{n_{1}\times n_{2}\times n_{3}}$};
			\end{tikzpicture}};
		\node[scale=2.7] at (6,0.5){
			\begin{tikzpicture}
			\draw[fill = white] (0,1.8) rectangle ++(0.6,0.6);
			\draw[fill = white] (0,1.2) rectangle ++(0.6,0.6);
			\draw[fill = white] (0,0) rectangle ++(0.6,0.6);
			\node[scale=0.5] at (0.25,0.9) {$\vdots$};
			\node[scale=0.35] at (0.3,0.25) {$\mathcal{A}(:,:,n_{3})$};
			\node[scale=0.35] at (0.3,1.45) {$\mathcal{A}(:,:,2)$};
			\node[scale=0.35] at (0.3,2.05) {$\mathcal{A}(:,:,1)$};
			\node[scale=0.37] at (0.3,-0.4) {$\mathrm{unfold}(\mathcal{A})$};
			\end{tikzpicture}
		};			
		\node[scale=1.5] at (2.5,0.7) {$\Rightarrow$};
	\end{tikzpicture}
\end{center}
\caption{The diagrammatic sketch of $\mathrm{unfold}( \mathcal{A})$ with $\mathcal{A}$ is an $n_{1}\times n_{2}\times n_{3}$ third-order tensor $\mathcal{A}$.}
\label{figure3}
\end{figure}
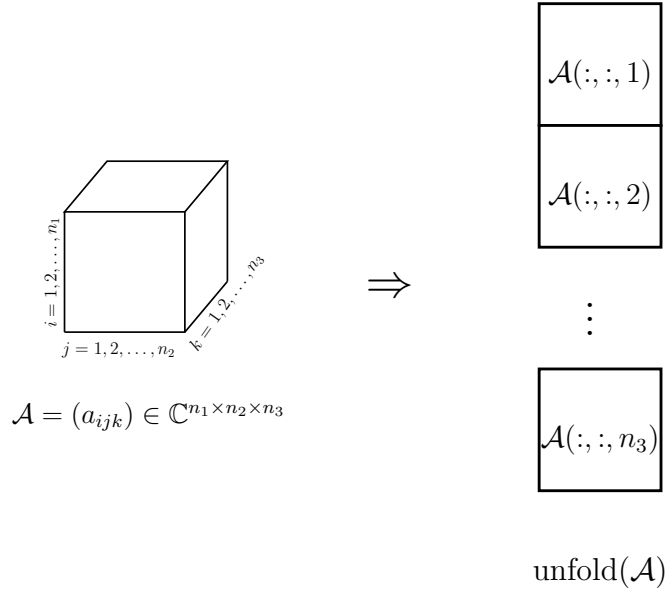
More specifically, we give an example of a $2\times 2 \times 2$ tensor in Figure \ref{figure4}.\\
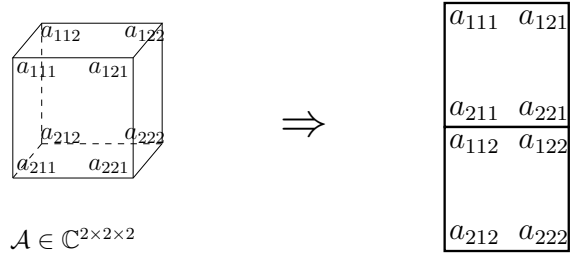
\begin{figure}[h] 
	\begin{center}
		\begin{tikzpicture}[scale=0.7]
		\node[scale=0.8] at (-2,0.5) {
			\begin{tikzpicture}
			\draw (0,0)coordinate(b) -- ++(2,0)coordinate(c) --++(50:0.75)coordinate(a) -- ++(-2,0)coordinate(d) -- cycle;
			\foreach \x in{a,b,c}
			\draw (\x) -- ++(0,-2);
			\draw (0,-2) -- ++(2,0) -- ++(50:0.75);
			\node[] at (0.4,-0.2) {$a_{111}$};
			\node[] at (1.6,-0.2) {$a_{121}$};
			\node[] at (0.4,-1.8) {$a_{211}$};
			\node[] at (1.6,-1.8) {$a_{221}$};
			\node[] at (0.8,0.4) {$a_{112}$};
			\node[] at (2.2,0.4) {$a_{122}$};
			\node[] at (0.8,-1.3) {$a_{212}$};
			\node[] at (2.2,-1.3) {$a_{222}$};
			\draw[dashed] (0.48,0.57)  -- (0.48,-1.43);
			\draw[dashed] (0.48,-1.43)  -- (2.47,-1.43);
			\draw[dashed]  (0.48,-1.43) -- (0,-2);
			\node[] at (1,-3) {$\mathcal{A}\in \mathbb{C}^{2\times2\times2}$};
			\end{tikzpicture}};
		\node[scale=2.2] at (6,0.7){
			\begin{tikzpicture}			
			\draw[fill = white] (0,1.5) rectangle ++(0.75,0.75);
			\draw[fill = white] (0,0.75) rectangle ++(0.75,0.75);	
			\node[scale=0.45] at (0.18,2.15) {$a_{111}$};
			\node[scale=0.45] at (0.6,2.15) {$a_{121}$};
			\node[scale=0.45] at (0.18,1.6) {$a_{211}$};
			\node[scale=0.45] at (0.6,1.6) {$a_{221}$};
			
			\node[scale=0.45] at (0.18,1.4) {$a_{112}$};
			\node[scale=0.45] at (0.6,1.4) {$a_{122}$};
			\node[scale=0.45] at (0.18,0.85) {$a_{212}$};
			\node[scale=0.45] at (0.6,0.85) {$a_{222}$};		
			\end{tikzpicture}
		};			
		\node[scale=1.5] at (2.1,0.7) {$\Rightarrow$};
	\end{tikzpicture}
\end{center}
\caption{The schematic of $\mathrm{unfold}( \mathcal{A})$ with a $2\times 2 \times 2$ tensor $\mathcal{A}$ .}
\label{figure4}
\end{figure}
\indent By using the unfold operator, we can create an $n_{1}n_{3} \times n_{2}n_{3}$ block-circulant matrix denoted by $\mathrm{circ}\left( \mathrm{unfold}\left( \mathcal{A}\right)\right)$ with the frontal slices of $\mathcal{A}$. It is given by 
\begin{equation*}
\mathrm{circ}\left( \mathrm{unfold}\left( \mathcal{A}\right)\right) := 
\left[
\begin{matrix}
\mathcal{A}\left( :,:,1\right) &\mathcal{A}\left( :,:,n_{3}\right) & \cdots & \mathcal{A}\left( :,:,2\right) \\
\mathcal{A}\left( :,:,2\right) & \mathcal{A}\left( :,:,1\right) & \cdots & \mathcal{A}\left( :,:,3\right) \\
\vdots &\vdots&  \ddots  & \vdots  \\
\mathcal{A}\left( :,:,n_{3}\right) &  \mathcal{A}\left( :,:,n_{3}{\small -1}\right)&\cdots & \mathcal{A}\left( :,:,1\right) 
\end{matrix}
\right].
\end{equation*}
Also, we denote the inverse operator of $\mathbf{circ}$ as $\mathbf{circ}^{-1}$ which works on block-circulant matrices. For example, we have
$$
\mathrm{circ}^{-1}(\mathrm{circ}\left( \mathrm{unfold}\left( \mathcal{A}\right)\right))=
\left[
\begin{matrix}
	\mathcal{A}\left( :,:,1\right)  \\
	\mathcal{A}\left( :,:,2\right)  \\
	\vdots   \\
	\mathcal{A}\left( :,:,n_{3}\right)  
\end{matrix}
\right].
$$

\indent Kilmer, Martin, and Perrone \cite{t-product} presented a new tensor multiplication called t-product. The t-product applies $\mathbf{fold}$, $\mathbf{unfold}$, and $\mathbf{circ}$ operators and requests that two third-order tensors have matching dimensions. This tensor multiplication can be described in the following.  
\begin{definition}	\label{definition 2.4}
	(T-product \cite{kilmer2011factorization,t-product})   
	Let $\mathcal{A}$ is an $n_{1}\times n_{2}\times n_{3}$ third-order tensor and $\mathcal{B}$ is an $n_{2}\times l \times n_{3}$ third-order tensor, then the product of  $\mathcal{A} \ast \mathcal{B} $ is an ${n_{1}\times l\times n_{3}}$
	tensor can be represented as
	\begin{equation*}
	\mathcal{A} \ast \mathcal{B}=\mathrm{fold}\left[ \mathrm{circ}\left(\mathrm{unfold\left( \mathcal{A}\right)}\right) \cdot \mathrm{unfold\left( \mathcal{B}\right)} \right],  
	\end{equation*}
\end{definition}
\noindent where $\ast$ denotes the new tensor multiplication of two third-order tensors called t-product, and $\cdot$ denotes the general matrix product.\\        
\noindent  \textbf{Example 2.2.}\normalsize ~ Suppose $\mathcal{A}$ is an $n_{1}\times n_{2}\times 3$ third-order tensor and $\mathcal{B}$ is an $n_{2}\times l \times 3$ third-order tensor, then
\begin{equation*}
\begin{aligned}
\mathcal{A} \ast \mathcal{B} &=\mathrm{fold}\left[ \mathrm{circ}\left(\mathrm{unfold\left( \mathcal{A}\right)}\right) \cdot \mathrm{unfold\left( \mathcal{B}\right)} \right] \\
&=\mathrm{fold}\left( 
\left[
\begin{matrix}
\mathcal{A}(:,:,1) & \mathcal{A}(:,:,3) & \mathcal{A}(:,:,2)\\
\mathcal{A}(:,:,2) & \mathcal{A}(:,:,1) & \mathcal{A}(:,:,3) \\
\mathcal{A}(:,:,3) & \mathcal{A}(:,:,2) & \mathcal{A}(:,:,1)\\   
\end{matrix}
\right] \cdot
\left[
\begin{matrix}
\mathcal{B}(:,:,1) \\
\mathcal{B}(:,:,2) \\
\mathcal{B}(:,:,3)\\   
\end{matrix}
\right]\right)
\end{aligned}
\end{equation*}
\noindent is an $n_{1}\times l\times 3$ tensor.\\          
\indent Using the knowledge of Fourier transform, the block-circulant matrix can be transformed into a block-diagonal matrix. Suppose that $F_{n_{3}}$ is the $ {n_{3}\times n_{3}} $  Fourier matrix. Then 
\begin{equation*}
\left( F_{n_{3}} \otimes I_{n_{1}}\right) \cdot  \mathrm{circ}\left( \mathrm{unfold\left( \mathcal{A}\right)}\right)  \cdot \left( {F_{n_{3}}^{H}} \otimes I_{n_{2}}\right)=
\begin{bmatrix}
\bar{A}_{1} &   & &  \\
& \bar{A}_{2} &  & \\
&   & \ddots &\\
&   &  &\bar{A}_{n_{3}}
\end{bmatrix}.
\end{equation*}\\
 Let
\begin{equation*}
\bar{A}:=
\begin{bmatrix}
\bar{A}_{1} &   & &  \\
& \bar{A}_{2} &  & \\
&   & \ddots &\\
&   &  &\bar{A}_{n_{3}}
\end{bmatrix}.
\end{equation*}
Then $\bar{A}$ is a complex $n_{1}n_{3}\times n_{2}n_{3}$ block-diagonal matrix. Let $\bar{A}_{i}$ be the $i$-th frontal slice of $\hat{\mathcal{A}}$, where $\hat{\mathcal{A}}\in \mathbb{C}^{n_{1}\times n_{2} \times n_{3}}$ is the result of discrete Fourier transform (DFT) on $\mathcal{A}$ along the third dimension. The MATLAB command for computing $\hat{\mathcal{A}}$ is 
\begin{equation}\label{Ahat}
	\hat{\mathcal{A}}=\mathrm{fft}[\mathcal{A}, [\,], 3].
\end{equation}  
   
\indent The t-product between two tensors can be understood as the matrix multiplication in the Fourier domain \cite{tensorrobust}. The t-product result between $\mathcal{A}$ and $\mathcal{B}$ can be obtained by multiplying each pair of the frontal slices of $\hat{\mathcal{A}}$ and $\hat{\mathcal{B}}$ and computing the inverse Fourier transform along the third dimension, where $\hat{\mathcal{B}}$ is the result of DFT on $\mathcal{B}$ along the third dimension, the MATLAB commands is  $\hat{\mathcal{B}}$=fft[$\mathcal{B}$, [\,], 3]. \\
\indent Next, we will introduce some basic knowledge of tensors, which can be found in \cite{kilmer2013third, tensorrobust}. 
\begin{definition} \label{definition 2.5}
	(F-diagonal Tensor) Suppose $\mathcal{A}\in \mathbb{C}^{n_{1}\times n_{2}\times n_{3}}$ is a third-order tensor, then $\mathcal{A}$ is the called f-diagonal tensor if its each frontal slice is a diagonal matrix.
\end{definition}
\begin{definition} \label{definition 2.6}
	(Identity Tensor) The ${n\times n\times k}$ tensor $\mathcal{I}_{nnk}$ is called identity tensor if its first frontal slice is an $n\times n$ identity matrix and other frontal slices are all zeros. If k=1, then the $n\times n\times 1$ tensor $\mathcal{I}_{nn1}$ is a special identity tensor with the frontal face being the $n\times n$ identity matrix, and we use $\mathcal{I}_{n}$ to denote it. 
\end{definition} 
\begin{definition} \label{definition 2.12}
	(Tensor Transpose) If $\mathcal{A} \in \mathbb{C}^{n_{1}\times n_{2}\times n_{3}}$ is a third-order tensor, then we use $\mathcal{A}^{T}$ to denote the transpose of tensor $\mathcal{A}$, which is an $n_{2}\times n_{1}\times n_{3}$ tensor obtained
	by transposing each of the frontal slice and then reversing the order of transposed faces  $2$ through $n_{3}$. 
\end{definition}
\begin{definition} \label{definition 2.8}
	(Tensor Conjugate Transpose) If $\mathcal{A}$ is a complex $n_{1}\times n_{2}\times n_{3}$ tensor, then $\mathcal{A}^{H}$ is used to denote the conjugate transpose of tensor $\mathcal{A}$, which is an $n_{2}\times n_{1}\times n_{3}$ tensor obtained
	by conjugate transposing each of the frontal slice and then reversing the order of transposed faces  $2$ through $n_{3}$. 
\end{definition}
\noindent \textbf{Example 2.3.}\normalsize ~ Suppose $\mathcal{A}$ is a complex $n_{1}\times n_{2}\times 4$ tensor, and $\mathcal{A}_{1}$, $\mathcal{A}_{2}$, $\mathcal{A}_{3}$, $\mathcal{A}_{4}$ denote its four frontal slices. Then 
\begin{equation*}
\mathcal{A}^{H}=\mathrm{fold}\left( 
\left[  \begin{matrix}
\mathcal{A}_{1}^{H} \\
\mathcal{A}_{4}^{H} \\
\mathcal{A}_{3}^{H}\\ 
\mathcal{A}_{2}^{H}  
\end{matrix}\right] \right). 
\end{equation*}
\begin{definition} \label{definition 2.7}
	(Unitary Tensor) An ${n\times n\times k}$ complex tensor $\mathcal{U}$ is unitary if $$\mathcal{U} \ast \mathcal{U}^{H}=\mathcal{U}^{H} \ast \mathcal{U}=\mathcal{I}_{nnk}.$$
\end{definition}
\indent Based on these basic definitions, the T-SVD is presented as a third-order
generalization of SVD. This generalization can write a third-order tensor as products of three other third-order tensors.
\begin{theorem}(T-SVD \cite{tensorrobust})\label{theorem 2.1}
	Suppose $\mathcal{A}\in \mathbb{C}^{n_{1}\times n_{2}\times n_{3}}$ is a third-order tensor, then $\mathcal{A}$ can be factored as\\
	\begin{equation*}
	\mathcal{A}=\mathcal{U} \ast \mathcal{S} \ast \mathcal{V}^{H},
	\end{equation*}
	where $\mathcal{U}$, $\mathcal{V}$ are $n_{1}\times n_{1}\times n_{3}$ and  $n_{2}\times n_{2}\times n_{3}$ unitary tensors, respectively, and $\mathcal{S}$ is an $n_{1}\times n_{2}\times n_{3}$ f-diagonal tensor.
\end{theorem}
\noindent $\mathbf{Note}$ $\mathbf{1:}$ Suppose $\hat{\mathcal{A}}$ is given by (\ref{Ahat}), then T-SVD of $ \mathcal{A}\in \mathbb{C}^{n_{1}\times n_{2}\times n_{3}}$ can be obtained by using SVD on each frontal slice of $\hat{\mathcal{A}}$.

\section{A New Semi-tensor Product Definition of Tensors}\label{sec:3} 
\indent In Section 2, we have introduced some properties and related knowledge of tensors. In the next, we will present a new definition of semi-tensor multiplication of third-order tensors and discuss the desirable theoretical property of this new multiplication.\\
\indent We present the definition of Frobenius norm of tensors. 
\begin{definition} \label{definition 3.2}
	(Frobenius Norm of Tensors \cite{kolda2006multilinear})   Suppose $\mathcal{A}=\left(a_{ijk}\right) $ is an $n_{1}\times n_{2}\times n_{3}$ third-order tensor, then the Frobenius norm of $\mathcal{A}$ is\\
	\begin{equation*}
	\Vert \mathcal{A} \Vert_{F}=\sqrt{\sum_{i=1}^{n_{1}} \sum_{j=1}^{n_{2}}\sum_{k=1}^{n_{3}} \vert a_{ijk}\vert ^{2}} .
	\end{equation*}
\end{definition}
\indent Before we define the semi-tensor product of tensors, we need the knowledge of Kronecker product of tensors which is similar to the one of matrices.
\begin{definition} (Kronecker Product of Tensors \cite{batselier2017constructive, liu2022semi}) \label{definition 3.3}
	Suppose $\mathcal{A}=\left(a_{i_{1}i_{2}\cdots i_{p}}\right)$  is an $n_{1}\times n_{2}\times \cdots \times n_{p} $ $p$-th order tensor, and $\mathcal{B}=\left(b_{j_{1}j_{2}\cdots j_{p}}\right)$ is an $ m_{1}\times m_{2}\times \cdots \times m_{p}$ $p$-th order tensor. Then the Kronecker product of $\mathcal{A}$ and $\mathcal{B}$ is defined by\\
	\begin{equation*}
	\mathcal{A} \otimes \mathcal{B}	=\left[ a_{i_{1}i_{2}\cdots i_{p}} \right] \mathcal{B},
	\end{equation*}
	which is an $n_{1}m_{1}\times n_{2}m_{2}\times \cdots \times n_{p}m_{p}$ $p$-th order tensor, and whose each entry of it can be represented as 
	\begin{equation*}
	\mathcal{A} \otimes \mathcal{B}_{\left[i_{1}j_{1}\right]\left[i_{2}j_{2}\right] \cdots \left[i_{p}j_{p}\right]}	= a_{i_{1}i_{2}\cdots i_{p}} b_{j_{1}j_{2}\cdots j_{p}}.
	\end{equation*}
\end{definition}
We briefly list some properties of the tensor Kronecker product.
\begin{lemma} (\cite{batselier2017constructive}) \label{lemma 3.1}
	Suppose $\mathcal{A}$, $\mathcal{B}$, $\mathcal{C}$ are $p$-th order tensors and $\alpha$ is a scalar, then\\		
	$(\rmnum{1})$  $\mathcal{A} \otimes\left( \mathcal{B}\pm \mathcal{C}\right) =(\mathcal{A}\otimes\mathcal{B})\pm(\mathcal{A}\otimes \mathcal{C})$, $\left( \mathcal{B}\pm \mathcal{C} \right) \otimes \mathcal{A} = \mathcal{B} \otimes \mathcal{A} \pm \mathcal{C}\otimes A;$\\
	$(\rmnum{2})$  $\left(  \mathcal{A}  \otimes \mathcal{B} \right)\otimes \mathcal{C} = \mathcal{A} \otimes (\mathcal{B} \otimes \mathcal{C});$ \\
	$(\rmnum{3})$  $\left(\alpha \mathcal{A} \right) \otimes \mathcal{B}  = \mathcal{A} \otimes \left(\alpha\mathcal{B}\right)=\alpha \left( \mathcal{A} \otimes \mathcal{B}\right).$ 
\end{lemma}
\indent After giving the relevant knowledge above, we give a new definition of semi-tensor product of tensors, which requires that tensors have matching dimensions. We mainly give the definition of third-order tensors in this paper.
\begin{definition}    \label{definition 3.4}
	(Semi-tensor Product of Tensor) Let $\mathcal{I}_{k}$ be the $k\times k \times 1$ identity tensor defined in Definition \ref{definition 2.6},  $\mathcal{A}=\left(a_{i_{1}i_{2}i_{3}}\right) $ be an $ m\times n\times t $ third-order tensor, and $\mathcal{B}=\left(b_{j_{1}j_{2}j_{3}}\right) $ be a $ p\times q\times t $ third-order tensor. We have \\
	$(\rmnum{1})$ If $n=kp,k \in \mathbb{Z}^{+}$, then\\
	\begin{equation*}
	\begin{aligned}
	\mathcal{A} \ltimes \mathcal{B} &=\mathcal{A} \ast \left( \mathcal{B} \otimes \mathcal{I}_{k} \right)\\
	&=\mathrm{fold}\left[ \mathrm{circ}\left( \mathrm{unfold(\mathcal{A})}\right) \cdot \mathrm{unfold}\left( \mathcal{B} \otimes \mathcal{I}_{k}\right)\right]
	\end{aligned}
	\end{equation*} 
	is an $m\times kq\times t$ third-order tensor;\\
	$(\rmnum{2})$ If $n=\frac{1}{k}p,k \in \mathbb{Z}^{+}$, then\\
	\begin{equation*}
	\begin{aligned}
	\mathcal{A} \ltimes \mathcal{B} &=\left( \mathcal{A} \otimes \mathcal{I}_{k} \right) \ast \mathcal{B}  \\
	&=\mathrm{fold}\left[ \mathrm{circ}\left( \mathrm{unfold \left( \mathcal{A} \otimes \mathcal{I}_{k}\right)}\right) \cdot \mathrm{unfold}\left( \mathcal{B}\right) \right]
	\end{aligned}
	\end{equation*} 
	is a $km\times q\times t$ third-order tensor. 
\end{definition}	
\indent According to Definition \ref{definition 2.4}, we can see that if $n=p$, then the semi-tensor product of tensors is actully the t-product. For this definition of tensor semi-tensor product, we can get another equivalent representation as follows.
\begin{lemma} \label{lemma 3.2}
	Suppose  $\mathcal{A}=\left(a_{i_{1}i_{2}i_{3}}\right) $  is an  $ m\times n\times t $ third-order tensor and $\mathcal{B}=\left(b_{j_{1}j_{2}j_{3}}\right) $ is a $ p\times q\times t $ third-order tensor. \\
	$(\rmnum{1})$ If $n=kp,k \in \mathbb{Z}^{+}$, then\\
	\begin{equation*}
	\begin{aligned}
	\mathcal{A} \ltimes \mathcal{B} &=\mathrm{fold}\left[ \mathrm{circ}\left( \mathrm{unfold(\mathcal{A})}\right) \cdot  \mathrm{unfold}\left( \mathcal{B} \otimes \mathcal{I}_{k}\right)\right]  \\
	&=\mathrm{fold}\left[ \mathrm{circ}\left( \mathrm{unfold(\mathcal{A})}\right) \ltimes \mathrm{unfold}\left( \mathcal{B} \right)\right]
	\end{aligned}
	\end{equation*}  
	is an $m\times kq\times t$ third-order tensor; \\
	$(\rmnum{2})$ If $n=\frac{1}{k}p,k \in \mathbb{Z}^{+}$, then\\
	\begin{equation*}
	\begin{aligned}
	\mathcal{A} \ltimes \mathcal{B}&=\mathrm{fold}\left[ \mathrm{circ}\left( \mathrm{unfold \left( \mathcal{A} \otimes \mathcal{I}_{k}\right)}\right) \cdot \mathrm{unfold}\left( \mathcal{B}\right)\right]  \\
	&=\mathrm{fold}\left[ \mathrm{circ}\left( \mathrm{unfold \left( \mathcal{A} \right)}\right) \ltimes \mathrm{unfold}\left( \mathcal{B}\right)\right]
	\end{aligned}
	\end{equation*}  
	is a $km\times q\times t$ third-order tensor.
\end{lemma} 
\begin{proof}
	From Lemma \ref{lemma 2.2}, Definitions \ref{definition 3.3} and \ref{definition 3.4}, and definitions of fold and unfold operators, we can prove that \\
	$(\rmnum{1})$ 
	\begin{equation*}
	\begin{aligned}
	\mathcal{A} \ltimes \mathcal{B} &=\mathrm{fold}\left[ \mathrm{circ}\left( \mathrm{unfold(\mathcal{A})}\right) \cdot \mathrm{unfold}\left( \mathcal{B} \otimes \mathcal{I}_{k}\right)\right]  \\
	&=\mathrm{fold}\left[ \mathrm{circ}\left( \mathrm{unfold(\mathcal{A})}\right)\cdot ( \mathrm{unfold}\left( \mathcal{B}\right) \otimes {I}_{k})\right]  \\
	&=\mathrm{fold}\left[ \mathrm{circ}\left( \mathrm{unfold(\mathcal{A})}\right) \ltimes \mathrm{unfold}\left( \mathcal{B} \right) \right],
	\end{aligned}
	\end{equation*}      	
	and\\
	$(\rmnum{2})$ 	
	\begin{equation*}
	\begin{aligned}
	\mathcal{A} \ltimes \mathcal{B}&=\mathrm{fold}\left[ \mathrm{circ}\left( \mathrm{unfold \left( \mathcal{A} \otimes \mathcal{I}_{k}\right)}\right) \cdot \mathrm{unfold}\left( \mathcal{B}\right) \right] \\
	&=\mathrm{fold}\left[ \mathrm{circ}\left( \left( \mathrm{unfold}  \left( \mathcal{A} \right) \right) \otimes {I}_{k}\right) \cdot  \mathrm{unfold}\left( \mathcal{B}\right)\right] \\
	&=\mathrm{fold}\left[ \mathrm{circ}\left( \mathrm{unfold \left( \mathcal{A} \right)}\right) \ltimes \mathrm{unfold}\left( \mathcal{B}\right) \right], 
	\end{aligned}
	\end{equation*} 
	where ${I}_{k}$ is the $k\times k$ identity matrix and $\mathcal{I}_{k}$ is the $k\times k\times 1$ identity tensor.
\end{proof} 
\indent  Next, we give a basic property of semi-tensor product of tensors.
\begin{theorem} \label{theorem 3.1}
	Let $\mathcal{A}$, $\mathcal{B}$, and $\mathcal{C}$ are three third-order tensors, they have proper dimensions such that the semi-tensor product is well-defined, then we have \\
	\begin{equation*}
	\left( \mathcal{A} \ltimes \mathcal{B}\right) 	\ltimes \mathcal{C} =    \mathcal{A} \ltimes \left( \mathcal{B}	\ltimes \mathcal{C}	\right).
	\end{equation*} 
\end{theorem}
\begin{proof}
	From Lemma \ref{lemma 3.2}, we have\\
	\begin{equation*}
	\mathcal{A} \ltimes \mathcal{B} =\mathrm{fold}\left[ \mathrm{circ}\left( \mathrm{unfold(\mathcal{A})}\right) \ltimes \mathrm{unfold}\left( \mathcal{B} \right)\right],	   
	\end{equation*}
	\begin{equation*}
	\mathcal{B} \ltimes \mathcal{C} =\mathrm{fold}\left[ \mathrm{circ}\left( \mathrm{unfold(\mathcal{B})}\right) \ltimes \mathrm{unfold}\left( \mathcal{C} \right)\right].
	\end{equation*}
	Then,\\
	\begin{equation*}
	\begin{aligned}
	&(\mathcal{A} \ltimes \mathcal{B})\ltimes \mathcal{C}\\
	=&  \mathrm{fold}[\mathrm{circ}[\mathrm{unfold}[\mathrm{fold}(\mathrm{circ}(\mathrm{unfold}(\mathcal{A}))\ltimes \mathrm{unfold}(\mathcal{B}))]]\ltimes \mathrm{unfold}(\mathcal{C})]\\
	=& \mathrm{fold}[\mathrm{circ}[\mathrm{circ}(\mathrm{unfold}(\mathcal{A}))\ltimes \mathrm{unfold}(\mathcal{B})]\ltimes\mathrm{unfold}(\mathcal{C})]
	\end{aligned}
	\end{equation*}
	and 
	\begin{equation*}
	\begin{aligned}
	&\mathcal{A} \ltimes \left( \mathcal{B}\ltimes \mathcal{C}\right)\\
	=& \mathrm{fold}[\mathrm{circ}(\mathrm{unfold}( \mathcal{A}))\ltimes \mathrm{unfold}[\mathrm{fold}(\mathrm{circ}(\mathrm{unfold}( \mathcal{B} ))\ltimes \mathrm{unfold}(\mathcal{C})]] \\
	=&\mathrm{fold}[\mathrm{circ}(\mathrm{unfold}(\mathcal{A}))\ltimes (\mathrm{\mathrm{circ}}(\mathrm{unfold}(\mathcal{B}))\ltimes \mathrm{unfold}(\mathcal{C})].
	\end{aligned} 
	\end{equation*} 
	Thus, we only need to prove \\
	\begin{equation} \label{equation 3.1}
	\begin{aligned}
	\mathrm{circ}[\mathrm{circ}(\mathrm{unfold}(\mathcal{A}))\ltimes \mathrm{unfold}(\mathcal{B})]\ltimes \mathrm{unfold}(\mathcal{C}) \\= \mathrm{circ}(\mathrm{unfold}(\mathcal{A}))\ltimes (\mathrm{circ}(\mathrm{unfold}(\mathcal{B}))\ltimes\mathrm{unfold}(\mathcal{C}).
	\end{aligned} 
	\end{equation}
	The left side of (\ref{equation 3.1}) is equivalent to\\
	\begin{equation}	\label{equation 3.2} 
	\begin{aligned}
	(\mathrm{circ}(\mathrm{unfold}(\mathcal{A}))\ltimes \mathrm{circ}(\mathrm{unfold}(\mathcal{B}))) \ltimes \mathrm{unfold}(\mathcal{C}).	
	\end{aligned}      
	\end{equation}
	By applying Lemma \ref{lemma 2.3}, we can see that (\ref{equation 3.2}) can be written as\\
	\begin{equation*}
	\begin{aligned}
	(\mathrm{circ}(\mathrm{unfold}(\mathcal{A}))\ltimes \mathrm{circ}(\mathrm{unfold}(\mathcal{B}))) \ltimes \mathrm{unfold}(\mathcal{C}) \\= \mathrm{circ}(\mathrm{unfold}(\mathcal{A}))\ltimes (\mathrm{circ}(\mathrm{unfold}(\mathcal{B}))\ltimes \mathrm{unfold}(\mathcal{C}).
	\end{aligned} 
	\end{equation*}
	 Hence the theorem is proved.
\end{proof}
\indent From Lemma \ref{lemma 3.2}, we know that the semi-tensor product between two tensors can be understood as the matrix multiplication in the Fourier domain \cite{tensorrobust}. \\
\indent Suppose $\mathcal{A}$ is an $ m\times n\times t $ third-order tensor and $\mathcal{B}$ is a $ p\times q\times t $ third-order tensor. If $n=kp, k \in \mathbb{Z}^{+}$, then $\mathcal{C}=\mathcal{A} \ltimes \mathcal{B}$ is equivalent to\\
\begin{equation*}
\mathrm{unfold\left( \mathcal{C}\right)}= \mathrm{circ}\left( \mathrm{unfold\left( \mathcal{A}\right)}\right) \ltimes \mathrm{unfold\left( \mathcal{B}\right)},    
\end{equation*}
and we have
\begin{equation*}
\begin{aligned}
\mathrm{unfold(\hat{ \mathcal{C}})}=& F_{t}  \ltimes \mathrm{unfold\left( \mathcal{C}\right)}\\
=&F_{t}  \ltimes \mathrm{circ}\left( \mathrm{unfold\left( \mathcal{A}\right)}\right) \ltimes \mathrm{unfold\left( \mathcal{B}\right)}\\
=& F_{t} \ltimes \mathrm{circ}\left( \mathrm{unfold\left( \mathcal{A}\right)}\right) \ltimes {F_{t}^{H}} \ltimes F_{t} \ltimes \mathrm{unfold\left( \mathcal{B}\right)}  \\
=&\bar{A} \ltimes  \mathrm{unfold(\hat{ \mathcal{B}})}.
\end{aligned} 
\end{equation*}
Therefore, the semi-tensor product result between $\mathcal{A}$ and $\mathcal{B}$ can be obtained by multiplying each pair of the frontal slices of $\hat{\mathcal{A}}$ and $\hat{\mathcal{B}}$ by semi-tensor product and computing the inverse Fourier transform along the third dimension,where $\hat{\mathcal{B}}$ and $\hat{\mathcal{C}}$ are the results of DFT on $\mathcal{B}$ and $\mathcal{C}$ along the third dimension, the MATLAB commands are $\hat{\mathcal{B}}$=fft[$\mathcal{B}$, [\,], 3] and $\hat{\mathcal{C}}$=fft[$\mathcal{C}$, [\,], 3]. \\
\indent We note that Definitions \ref{definition 2.5}, \ref{definition 2.6}, \ref{definition 2.12}, \ref{definition 2.8}, and \ref{definition 2.7} in Section 2 still hold under this new multiplication.

\section{Tensor Decomposition Based on Semi-tensor Product}\label{sec:4}
\indent In this section, we will give a new approximate tensor decomposition model based on the semi-tensor product. 
\subsection{Singular Value Decomposition of Matrices based on Semi-tensor Product}	
\indent We first introduce a strategy of finding $B $ and $C $ so that $\Vert  A-B \otimes C \Vert _{F}$ is minimized \cite{van1993approximation}. Let \\
\begin{equation*}
A= \left[  
\begin{matrix}
A_{1,1} & \cdots & A_{1,n_{1}}\\
A_{2,1} & \cdots & A_{2,n_{1}} \\
\vdots  & \ddots & \vdots\\ 
A_{m_{1},1} & \cdots & A_{m_{1},n_{1}} 
\end{matrix}\right] \in \mathbb{C}^{m_{1}m_{2}\times n_{1} n_{2}} ,
\end{equation*}
where each $A_{i,j} (i=1, 2, \cdots, m_{1}, j=1, 2, \cdots, n_{1})$ is an $m_{2}\times n_{2}$ matrix. Let $\widetilde{A} \in \mathbb{C}^{m_{1}n_{1}\times m_{2} n_{2}} $ be defined by
\begin{equation} \label{equation 4.1}
\widetilde{A}:= \left[  
\begin{matrix}
\check{A}_{1} &\check{A}_{2} & 	\cdots &
\check{A} _{n_{1}}
\end{matrix}\right]^{T} 
\end{equation}
with
\begin{equation} \label{equation 4.2}
\check{A}_{j}= \left[  
\begin{matrix}
vec(A_{1,j})^{T}\\
vec(A_{2,j})^{T} \\
\vdots \\ 
vec(A_{m_{1},j})^{T}
\end{matrix}\right].
\end{equation} 
\begin{lemma}(\cite{van1993approximation})\label{lemma 4.1}
	Suppose $A\in \mathbb{C}^{m\times n} $ with $m=m_{1}m_{2}$ and $n=n_{1}n_{2}$, then we can find  $B \in \mathbb{C}^{m_{1}\times n_{1}} $ and $C\in \mathbb{C}^{m_{2}\times n_{2}} $ defined by $vec(B) =\sqrt{\widetilde{\sigma}_{1}} U\left( :,1\right) $ and $vec(C) =\sqrt{\widetilde{\sigma}_{1}} V\left( :,1\right) $ minimize
	\begin{equation*}
	\Vert A-B \otimes C \Vert _{F} ,
	\end{equation*} 
	where $\widetilde{\sigma}_{1}$ is the largest singular value of  $\widetilde{A}$ which is defined by (\ref{equation 4.1}), and $U\left( :,1\right) $, $V\left( :,1\right) $ are its corresponding left and right singular vectors, respectively.
\end{lemma}
\indent By using Lemma \ref{lemma 4.1}, we can find a SVD-like approximate decomposition based on semi-tensor product of matrix called STP-SVD.
\begin{theorem} \label{theorem 4.1}
	Suppose  $A\in \mathbb{C}^{m\times n} $ with $m=m_{1}m_{2}$ and $n=n_{1}n_{2}$. Then there exist unitary matrices $U \in \mathbb{C}^{m_{1}\times m_{1}} $ and $V \in \mathbb{C}^{n_{1}\times n_{1}} $ such that 
	\begin{equation} \label{equation 4.4}
	A= U \ltimes  \mathrm{\Sigma} \ltimes  V^{H} +E_{1},
	\end{equation} 
	where $\mathrm{\Sigma}=\mathrm{blkdiag}(S_{1}, S_{2}, \dots, S_{p}) \in \mathbb{C}^{m \times n}$ with each block being $ S_{i} \in \mathbb{C}^{m_{2}\times n_{2}}$, $i=1, 2, \dots ,p$, and $ p=min\left\lbrace m_{1},n_{1}\right\rbrace $, then $\Vert S_{1} \Vert _{F}\geqslant \Vert S_{2} \Vert _{F} \geqslant  \cdots \geqslant \Vert S_{p} \Vert _{F}  $. The matrix $E_{1}$ is the approximation error. 
\end{theorem}
\begin{proof}
	\indent From Lemma \ref{lemma 4.1}, we know that for $A\in \mathbb{C}^{m\times n} $ with $m=m_{1}m_{2}$ and $n=n_{1}n_{2}$, we can find  $B \in \mathbb{C}^{m_{1}\times n_{1}} $ and $C \in \mathbb{C}^{m_{2}\times n_{2}} $ minimize $	\Vert A-B \otimes C \Vert _{F}$. In other words, $A$ can be represented by  $B$ and $ C$ as
	\begin{equation} \label{equation 4.5}
   A = B \otimes C + E_{1}.
	\end{equation}
	Next, we can obtain $B= U {\Sigma}_{B}  V^{H}$ by computing the SVD of $B$, where $U\in \mathbb{C}^{m_{1}\times m_{1}} $, $V\in \mathbb{C}^{n_{1}\times n_{1}}$, and ${\Sigma}_{B}\in \mathbb{C}^{m_{1}\times n_{1}}$. Then, (\ref{equation 4.5}) can be rewritten as
	\begin{equation}  \label{equation 4.6}
	A = (U  {\Sigma}_{B} V^{H}) \otimes C + E_{1},
	\end{equation}
	where ${\Sigma}_{B}$ is a diagonal matrix, and the diagonal elements  of $\mathrm{\Sigma}_{B}$ are singular values of $B$. We use $\sigma_{1},\sigma_{2}, \dots, \sigma_{p}$,  $p=min\left\lbrace m_{1},n_{1}\right\rbrace $ to denote singular values of $B$, and $\sigma_{1}\geqslant\sigma_{2}\geqslant \dots \geqslant \sigma_{p}$, then $\mathrm{\Sigma}_{B}=diag(\sigma_{1},\sigma_{2}, \dots, \sigma_{p})$.\\
	\indent By applying Lemmas \ref{lemma 2.1} and \ref{lemma 2.2}, we can write (\ref{equation 4.6}) as
	\begin{equation*}  
	\begin{aligned}
	A &=\left( U  \mathrm{\Sigma}_{B}  V^{H}\right)  \otimes \left( I_{m_{2}} C I_{n_{2}}\right) +E_{1} \\
	&=\left( U  \otimes I_{m_{2}}\right) \left( \mathrm{\Sigma}_{B}  \otimes  C\right) \left( V^{H} \otimes  I_{n_{2}}\right) +E_{1}  \\
	&= U \ltimes  \mathrm{\Sigma} \ltimes  V^{H} +E_{1},
	\end{aligned}
	\end{equation*}
	where $\mathrm{\Sigma} =\mathrm{\Sigma}_{B}  \otimes C$  is an ${m_{1}m_{2}\times n_{1}n_{2}} $ block-diagonal matrix. And the diagonal elements of $\mathrm{\Sigma}$ are $S_{1}=\sigma_{1}C, S_{2}=\sigma_{2}C, \dots, S_{p}=\sigma_{p}C$. \\
	\indent Since $\mathrm{\sigma}_{i} (i=1,2, \dots, p)$ is the nonnegative number, then using the knowledge of norm, we have $\Vert \mathrm{\sigma}_{i} C \Vert _{F}=\mathrm{\sigma}_{i}\Vert C\Vert _{F}$, thus
	\begin{equation*}
	\Vert\sigma_{1} C \Vert _{F}\geqslant \Vert \sigma_{2}C\Vert _{F}\geqslant \dots \geqslant \Vert  \sigma_{p}C\Vert _{F},
	\end{equation*}
	which is
	\begin{equation*}
	\Vert S_{1}\Vert _{F}\geqslant \Vert S_{2}\Vert _{F}\geqslant \dots \geqslant \Vert  S_{p}\Vert _{F}.
	\end{equation*}
	\indent$E_{1}$ is the error matrix and $ \Vert E_{1}\Vert _{F}$ is the upper bound for approximation error $E_{1}$, then the proof is complete . 
\end{proof}	
\indent We can give a specific error analysis for the error matrix $E_{1}$ in Theorem \ref{theorem 4.1}. Suppose  $\tilde{\sigma}_{1},\tilde{\sigma}_{2}, \dots, \tilde{\sigma}_{q}$,  $q=min\left\lbrace m_{1}n_{1},m_{2}n_{2}\right\rbrace $ are singular values of $\widetilde{A} \in \mathbb{C}^{ m_{1}n_{1}\times m_{2}n_{2}} $ in (\ref{equation 4.1}), then
\begin{equation} \label{equation 4.7}
\begin{aligned}
\Vert E_{1}\Vert _{F}&=\sqrt{\tilde{\sigma}_{2}^{2}+\tilde{\sigma}_{3}^{2}+ \dots+ \tilde{\sigma}_{q}^{2}}\\
&=\sqrt{\sum_{i=2}^{q}\tilde{\sigma}_{i}^{2}}.
\end{aligned}
\end{equation} 
\indent We give the corresponding MATLAB preudocode in the following.
\begin{algorithm}[H]  
	\renewcommand{\algorithmicrequire}{\textbf{Input:}}
	\renewcommand{\algorithmicensure}{\textbf{Output:}}
	\caption{STP-SVD of Matrices}
	\begin{algorithmic}[1]
		\REQUIRE $A\in \mathbb{C}^{m\times n } $, $m_{2},n_{2}$. $(m=m_{1}m_{2}, n=n_{1}n_{2})$
		\ENSURE  $U, \mathrm{\Sigma}, V$.
		\STATE  By using Lemma 4.1, compute $B \in \mathbb{C}^{m_{1}\times n_{1}} $ and $C \in \mathbb{C}^{m_{2}\times n_{2}}$, such that  $A \approx B \otimes C $;
		\STATE Compute the SVD of $B$, i.e. $[U, \mathrm{\Sigma}_{B}, V]$=svd($B$), where $U\in \mathbb{C}^{m_{1}\times m_{1}} $, $V\in \mathbb{C}^{n_{1}\times n_{1}}$, and\\ $\mathrm{\Sigma}_{B}\in \mathbb{C}^{m_{1}\times n_{1}}$;
		\STATE $\mathrm{\Sigma} =\mathrm{\Sigma}_{B}  \otimes  C \in \mathbb{C}^{m_{1}m_{2}\times n_{1}n_{2}} $.		
	\end{algorithmic} \label{Algorithm 1}
\end{algorithm}  
\indent If we use the truncated SVD for $B$ in Algorithm \ref{Algorithm 1}, then we can obtain a corresponding truncated STP-SVD algorithm of matrices. The MATLAB preudocode is given as follows.  
\begin{algorithm}[H]  
	\renewcommand{\algorithmicrequire}{\textbf{Input:}}
	\renewcommand{\algorithmicensure}{\textbf{Output:}}
	\caption{Truncated STP-SVD of Matrices}
	\begin{algorithmic}[1]
		\REQUIRE $A\in \mathbb{C}^{m\times n } $, $m_{2},n_{2}$, r. $(m=m_{1}m_{2}, n=n_{1}n_{2})$
		\ENSURE  $U, \mathrm{\Sigma}, V$.
		\STATE  Compute $B \in \mathbb{C}^{m_{1}\times n_{1}} $ and $C \in \mathbb{C}^{m_{2}\times n_{2}}$, such that  $A \approx B \otimes C $;
		\STATE Compute the truncated SVD of $B$, i.e. $[U, \mathrm{\Sigma}_{B}, V]$=svds($B$, r), where $U\in \mathbb{C}^{m_{1}\times r} $, $V\in \mathbb{C}^{r \times n_{1}}$,\\ and $\mathrm{\Sigma}_{B}\in \mathbb{C}^{r\times r}$;
		\STATE $\mathrm{\Sigma} =\mathrm{\Sigma}_{B}  \otimes  C \in \mathbb{C}^{rm_{2}\times rn_{2}} $.		
	\end{algorithmic} \label{Algorithm 2}
\end{algorithm}  
\indent We write the truncated STP-SVD of matrices in a form similar to Theorem \ref{theorem 4.1}, i.e. 
\begin{equation} \label{equation 4.8}
A= U\ltimes  \mathrm{\Sigma} \ltimes  V^{H} +E.
\end{equation} 
\indent From the proof of Theorem \ref{theorem 4.1}, we can give a new error upper bound for truncated STP-SVD of the matrix. The error matrix $E$ in (\ref{equation 4.8}) can be separated into two parts, the first part is denoted as $E_{1}$ which produced by using rank-one SVD decomposition for $\widetilde{A}$, and the other part  $E_{2}$ comes from the truncated SVD of $B$.\\
\indent We have known the error upper bound for $ E_{1}$ from (\ref{equation 4.7}). Next, we give an error upper bound for $ E_{2} $ according to the proof of Theorem \ref{theorem 4.1}. Suppose the truncation at $r$ when calculating SVD on $B\in \mathbb{C}^{m_{1}\times n_{1}} $, then the error matrix is 
\begin{equation*}
 E_{2}=U \ltimes \Sigma \ltimes  V^{H}- U \ltimes \Sigma_{r} \ltimes  V^{H},
\end{equation*}  
where $\Sigma_{r}$ is obtained by preserving the first $r$ blocks of $\Sigma$ and changing the rest to $0$, which can also be represented by blkdiag$(S_{1}, \cdots, S_{r}, 0, \cdots, 0)$. We use $\Sigma_{p}$ to denote $\Sigma-\Sigma_{r}$, and $\Sigma_{p}$ therefore can be represented by blkdiag$(0, \cdots, 0, S_{r+1}, \cdots, S_{p})$, $p=\min\left\lbrace m_{1},n_{1}\right\rbrace $. Then we can obtain the upper bound for $E_{2}$: 
\begin{equation} \label{4.9}
\begin{aligned}
 E_{2}&=\Vert U \ltimes \Sigma \ltimes  V^{H}- U \ltimes \Sigma_{r} \ltimes  V^{H} \Vert _{F}	\\
 &=\Vert  U \ltimes \Sigma_{p} \ltimes  V^{H}    \Vert _{F}\\
 &=\Vert \mathrm{blkdiag}(0, \cdots, 0, S_{r+1}, \cdots, S_{p})  \Vert _{F} \\
&= \sqrt{\sum_{j=r+1}^{p}{\Vert \mathrm{S}_{j}\Vert}^{2}_{F}}.
\end{aligned}
\end{equation}
\indent Since the error matrix $ E$  in (\ref{equation 4.8}) can be divided into two parts $ E_{1}$ and $ E_{2}$. Then we can give the following upper bound for approximation $E$ from (\ref{equation 4.7}) and (\ref{4.9}). 
\begin{equation} \label{4.10}
\begin{aligned}
\Vert E \Vert _{F} \leq  \sqrt{\sum_{i=2}^{q}\tilde{\sigma}_{i}^{2}}+\sqrt{\sum_{j=r+1}^{p}{\Vert \mathrm{S}_{j}\Vert}^{2}_{F}},
\end{aligned}
\end{equation}\\
we know that $\tilde{\sigma}_{1},\tilde{\sigma}_{2}, \dots, \tilde{\sigma}_{q}$,  $q=\min\left\lbrace m_{1}n_{1},m_{2}n_{2}\right\rbrace $ are singular values of $	\widetilde{A} \in \mathbb{C}^{ m_{1}n_{1}\times m_{2}n_{2}} $ 
from the previous assumption.
\subsection{A New Tensor Decomposition Strategy based on Semi-tensor Product}
\indent From Theorem \ref{theorem 4.1}, we can learn that any matrix in the complex domain has a decomposition form based on the semi-tensor product. Besides, after giving the definition of tensor semi-tensor product, we can find that a tensor also can be decomposed based on semi-tensor product of tensors.   
\begin{theorem} \label{theorem 4.2}
	Suppose  $\mathcal{A}\in \mathbb{C}^{m\times n \times l} $ with $m=m_{1}m_{2}$ and $n=n_{1}n_{2}$ is a third-order tensor, then it can be factorized as 
	\begin{equation} \label{equation 4.10}
	\mathcal{A} = \mathcal{U}  \ltimes  \mathcal{S} \ltimes  \mathcal{V}^{H}+\mathcal{E}_{1}.
	\end{equation}
	where $\mathcal{U}$,  $\mathcal{V} $ are $m_{1}\times m_{1} \times l$ and $n_{1}\times n_{1} \times l$ unitary tensors, respectively, each frontal slice of $\mathcal{S}$ is a block-diagonal matrix, and $\mathcal{E}_{1}$ is an error tensor.
\end{theorem}	  	
\begin{proof} 
	By using the Fourier tranform, we suppose $F_{\mathit{l}}$ is an $ l\times l$  Fourier matrix, then we have
	\begin{equation} \label{equation 4.11}
	\left(F_{\mathit{l}} \otimes I_{m}\right) \mathrm{circ}\left( \mathrm{unfold\left( \mathcal{A}\right)}\right) \left( F_{\mathit{l}}^{H}\otimes I_{n}\right)=
	\begin{bmatrix}
	\bar{A}_{1} &   & &  \\
	& \bar{A}_{2} &  & \\
	&   & \ddots &\\
	&   &  &\bar{A}_{\mathit{l}}
	\end{bmatrix},	   	 
	\end{equation}
	\indent For each $ \bar{A}_{i}(i=1,2,\cdots, l)$, we have $ \bar{A}_{i} = \bar{U}_{i} \ltimes  \bar{\mathrm{\Sigma}}_{i} \ltimes  \bar{V}_{i}^{H} + \bar{E}_{i} $ by Theorem \ref{theorem 4.1}, the right side of (\ref{equation 4.11}) can be written as
	\begin{equation*}
	\begin{aligned}
	& \begin{bmatrix}
	\bar{A}_{1} &   & &  \\
	& \bar{A}_{2} &  & \\
	&   & \ddots &\\
	&   &  &\bar{A}_{\mathit{l}}
	\end{bmatrix}\\
	=& \begin{bmatrix}
	 \bar{U}_{1} \ltimes  \bar{ \mathrm{\Sigma}}_{1} \ltimes   \bar{V}_{1}^{H}+\bar{E}_{1} &   & &  \\
	&  \bar{U}_{2} \ltimes   \bar{\mathrm{\Sigma}}_{2} \ltimes  \bar{V}_{2}^{H}+\bar{E}_{2}&  & \\
	&   & \ddots &\\
	&   &  & \bar{U}_{\mathit{l}} \ltimes   \bar{\mathrm{\Sigma}}_{\mathit{l}} \ltimes   \bar{V}_{\mathit{l}}^{H}+\bar{E}_{l}
	\end{bmatrix}\\
	=&  \begin{bmatrix}
	 \bar{U}_{1} &   & &  \\
	&  \bar{U}_{2} &  & \\
	&   & \ddots &\\
	&   &  &  \bar{U}_{\mathit{l}} 
	\end{bmatrix} \ltimes
	\begin{bmatrix}
	 \bar{\mathrm{\Sigma}}_{1} &   & &  \\
	&  \bar{\mathrm{\Sigma}}_{2} &  & \\
	&   & \ddots &\\
	&   &  &  \bar{\mathrm{\Sigma}}_{\mathit{l}} 
	\end{bmatrix} \ltimes
	\begin{bmatrix}
	 \bar{V}_{1}^{H} &   & &  \\
	&  \bar{V}_{2}^{H} &  & \\
	&   & \ddots &\\
	&   &  &  \bar{V}_{\mathit{l}} ^{H}
	\end{bmatrix} +
		\begin{bmatrix}
	 \bar{E}_{1} &   & &  \\
	&  \bar{E}_{2} &  & \\
	&   & \ddots &\\
	&   &  &  \bar{E}_{\mathit{l}} 
	\end{bmatrix}. 
	\end{aligned}
	\end{equation*}
	\indent Next, we use semi-tensor product multiply the left and right sides of above formula by $ F_{\mathit{l}}^{H}  $ and $ F_{\mathit{l}}$, respectively, then we have
	\begin{equation*}
		\begin{aligned}
	&	F_{\mathit{l}}^{H}\ltimes
		\begin{bmatrix}
		\bar{A}_{1} &   & &  \\
		& \bar{A}_{2} &  & \\
		&   & \ddots &\\
		&   &  &\bar{A}_{\mathit{l}}
		\end{bmatrix} \ltimes  	F_{\mathit{l}}  \\	   	 	
		= &  
		F_{\mathit{l}}^{H} \ltimes
		\begin{bmatrix}
		 \bar{U}_{1} &   & &  \\
		& \bar{U}_{2} &  & \\
		&   & \ddots &\\
		&   &  & \bar{U}_{\mathit{l}} 
		\end{bmatrix} \ltimes
		\begin{bmatrix}
		\bar{\mathrm{\Sigma}}_{1} &   & &  \\
		& 	\bar{\mathrm{\Sigma}}_{2} &  & \\
		&   & \ddots &\\
		&   &  & 	\bar{\mathrm{\Sigma}}_{\mathit{l}} 
		\end{bmatrix} \ltimes
		\begin{bmatrix}
		\bar{V}_{1}^{H} &   & &  \\
		& \bar{V}_{2}^{H} &  & \\
		&   & \ddots &\\
		&   &  &\bar{V}_{\mathit{l}} ^{H} 
		\end{bmatrix} 
		\ltimes 
		F_{\mathit{l}}\\
	&	+
	F_{\mathit{l}}^{H}\ltimes
			\begin{bmatrix}
	 \bar{E}_{1} &   & &  \\
	&  \bar{E}_{2} &  & \\
	&   & \ddots &\\
	&   &  &  \bar{E}_{\mathit{l}} 
	\end{bmatrix}
		\ltimes 
		F_{\mathit{l}}\\
     	=&
     	F_{\mathit{l}}^{H} \ltimes
		\begin{bmatrix}
		\bar{U}_{1} &   & &  \\
		& \bar{U}_{2} &  & \\
		&   & \ddots &\\
		&   &  & \bar{U}_{\mathit{l}} 
		\end{bmatrix} \ltimes F_{\mathit{l}} \ltimes
		F_{\mathit{l}}^{H} \ltimes	\begin{bmatrix}
		\bar{\mathrm{\Sigma}}_{1} &   & &  \\
		& \	\bar{\mathrm{\Sigma}}_{2} &  & \\
		&   & \ddots &\\
		&   &  & 	\bar{\mathrm{\Sigma}}_{\mathit{l}} 
		\end{bmatrix} \ltimes F_{\mathit{l}} \ltimes 
	F_{\mathit{l}}^{H}\\
	&\ltimes
				\begin{bmatrix}
		\bar{V}_{1}^{H} &   & &  \\
		& \bar{V}_{2}^{H} &  & \\
		&   & \ddots &\\
		&   &  & \bar{V}_{\mathit{l}} ^{H} 
		\end{bmatrix} 
	\ltimes 
			F_{\mathit{l}}+ F_{\mathit{l}}^{H}\ltimes
		\begin{bmatrix}
	 \bar{E}_{1} &   & &  \\
	&  \bar{E}_{2} &  & \\
	&   & \ddots &\\
	&   &  &  \bar{E}_{\mathit{l}} 
	\end{bmatrix}
		\ltimes 
		F_{\mathit{l}},
		\end{aligned}
		\end{equation*}
	where 
	\begin{equation*}
	\begin{aligned}
	U:=
	F_{\mathit{l}}^{H} \ltimes
	\begin{bmatrix}
	\bar{U}_{1} &   & &  \\
	& \bar{U}_{2} &  & \\
	&   & \ddots &\\
	&   &  & \bar{U}_{\mathit{l}} 
	\end{bmatrix} \ltimes F_{\mathit{l}} ,
	\end{aligned}
	\end{equation*}
	\begin{equation*}
	\begin{aligned}
	\mathrm{\Sigma}:=
	F_{\mathit{l}}^{H} \ltimes	\begin{bmatrix}
	\bar{\mathrm{\Sigma}}_{1} &   & &  \\
	& 	\bar{\mathrm{\Sigma}}_{2} &  & \\
	&   & \ddots &\\
	&   &  & 	\bar{\mathrm{\Sigma}}_{\mathit{l}} 
	\end{bmatrix} \ltimes F_{\mathit{l}}, 
	\end{aligned}
	\end{equation*}
	\begin{equation*}
	\begin{aligned}
	V^{H}:=
	F_{\mathit{l}}^{H} \ltimes
	\begin{bmatrix}
	\bar{V}_{1}^{H} &   & &  \\
	& \bar{V}_{2}^{H} &  & \\
	&   & \ddots &\\
	&   &  & \bar{V}_{\mathit{l}} ^{H} 
	\end{bmatrix} 
	\ltimes 
	F_{\mathit{l}},
	\end{aligned}
	\end{equation*}
	and 
	\begin{equation*}
	\begin{aligned}
	E_{1}:=
	F_{\mathit{l}}^{H} \ltimes
	\begin{bmatrix}
	\bar{E}_{1}^{H} &   & &  \\
	& \bar{E}_{2}^{H} &  & \\
	&   & \ddots &\\
	&   &  & \bar{E}_{\mathit{l}} ^{H} 
	\end{bmatrix} 
	\ltimes 
	F_{\mathit{l}}
	\end{aligned}
	\end{equation*}
	are block-circulant matrices. Let $\mathcal{U}=\mathrm{fold}(\mathrm{circ}^{-1}(U))$, $\mathcal{S}=\mathrm{fold}(\mathrm{circ}^{-1}(\mathrm{\Sigma}))$, $\mathcal{V}^{H}=\mathrm{fold}(\mathrm{circ}^{-1}(V^{H}))$, and $\mathcal{E}_{1}=\mathrm{fold}(\mathrm{circ}^{-1}(E_{1}))$, then we can obtain an approximate tensor decomposition for $\mathcal{A}$ of the form (\ref{equation 4.10}).\\
	\indent Since each $\bar{U}_{i}$ is unitary, we know that $U$ is also a unitary matrix and $\mathcal{U}=\mathrm{fold}(\mathrm{circ}^{-1}(U))$. From Definition \ref{definition 2.8}, we have $\mathcal{U}^{H}=\mathrm{fold}(\mathrm{circ}^{-1}(U^{H}))$, then
	\begin{equation*}
	    \begin{aligned}
	    \mathcal{U}^{H} \ltimes \mathcal{U}= \mathrm{fold}(U^{H} \ltimes \mathrm{circ}^{-1}(U))
	    =\mathrm{fold}\left(\left[
	    \begin{matrix}
	    I_{m_{1}}    \\
        0 \\
\vdots	\\
	0 
	    \end{matrix}\right] \right) = \mathcal{I}_{m_{1}m_{1}\mathit{l}},
	    \end{aligned}
	\end{equation*}
	which indicates $\mathcal{U}$ is a unitary tensor. Similarly, $\mathcal{V}$ is also unitary.\\
\indent	The tensor $\mathcal{E}_{1}$ is the approximation error of the decomposition in (\ref{equation 4.10}). According to the unitary invariance of Frobenius norm, we have 
	\begin{equation*}
	\Vert \mathcal{E}_{1}  \Vert_{F}  = \sqrt{ {\Vert \bar{E}_{1} \Vert}^{2}_{F}+  {\Vert\bar{E}_{2} \Vert}^{2} _{F}+ \dots + {\Vert \bar{E}_{l} \Vert}^{2}_{F}},
	\end{equation*}
	for the approximation error tensor $\mathcal{E}_{1}$, then the proof is complete.
\end{proof} 
\noindent $\mathbf{Note}$ $\mathbf{2:}$ Suppose $\hat{\mathcal{A}}$ is considered as the tensor obtained by using Fourier transform of $ \mathcal{A}$, then STP-SVD of tensor $ \mathcal{A}\in \mathbb{C}^{n_{1}\times n_{2}\times n_{3}}$ can be obtained by computing matrix STP-SVD on each frontal slice of $\hat{\mathcal{A}}$.  \\
\indent We give the MATLAB psuedocode in Algorithm \ref{algorithm 3} for this decomposition strategy. \\
\begin{algorithm}[H]
	\renewcommand{\algorithmicrequire}{\textbf{Input:}}
	\renewcommand{\algorithmicensure}{\textbf{Output:}}
	\caption{STP-SVD of Tensors}
	\begin{algorithmic}[1]  \label{algorithm 3}
		\REQUIRE $\mathcal{A}\in \mathbb{C}^{m_{1}m_{2}\times n_{1}n_{2}\times l} $, $m_{2},n_{2}$.
		\ENSURE  $\mathcal{U}, \mathcal{S}, \mathcal{V}$.
		\STATE  Perform Fourier transform on $\mathcal{A}$, i.e.$\hat{\mathcal{A}}=\mathrm{fft}(\mathcal{A}, [\,], 3)$;
		\STATE From Lemma 4.1, compute $B_{i} $ and $C_{i} $ and do SVD on $B_{i} $:\\
		$\mathbf{for}$ $i=1:l$ $\mathbf{do}$\\ 
		\quad $\hat{\mathcal{A}}(:,:,i) \approx B_{i} \otimes C_{i} $;\\
		\quad $[ U_{i}, \mathrm{\Sigma}_{{B}_{i}}, V_{i}]=\mathrm{svd}(B_{i} )$;\\
		\quad  $\mathrm{\Sigma}_{i} =\mathrm{\Sigma}_{{B}_{i}}  \otimes C_{i} \in \mathbb{C}^{m_{1}m_{2}\times n_{1}n_{2}} $;\\
		\quad $\mathcal{U}_{i} \leftarrow U_{i}$, $\mathcal{S}_{i} \leftarrow \mathrm{\Sigma}_{i}$, $\mathcal{V}_{i} \leftarrow V_{i}$;\\
		$\mathbf{end}$ $\mathbf{for}$
		\STATE $\mathcal{U}=\mathrm{ifft}(\mathcal{U}, [\,], 3)$, $\mathcal{S}=\mathrm{ifft}(\mathcal{S}, [\,], 3)$, $\mathcal{V}=\mathrm{ifft}(\mathcal{V}, [\,], 3)$.\\
		($\hat{\mathcal{A}}(:,:,i)$, $\mathcal{U}_{i}$, $\mathcal{S}_{i}$, and $\mathcal{V}_{i}$ denote the $i$-th frontal slice of $\hat{\mathcal{A}}$, $\mathcal{U}$, $\mathcal{S}$, and $\mathcal{V}$, respectively.)   	
	\end{algorithmic}
\end{algorithm}	 
\indent We have given the truncated STP-SVD of matrices, and now we give the truncated STP-SVD of tensors. The idea is using truncated SVD on $B_{i}$ every time we decompose $\hat{\mathcal{A}}(:,:,i)$. Given a vector $\mathbf{R}=[R_1,R_2,\dots,R_l]^T\in \mathbb{N_{+}}^{l}$, when we do STP-SVD on the $i$-th frontal slice of $\hat{\mathcal{A}}$, the SVD in Algorithm \ref{algorithm 3} is truncated at $\mathit{R}_i$. For the decomposition $\mathcal{A} = \mathcal{U}  \ltimes  \mathcal{S} \ltimes  \mathcal{V}^{H}+\mathcal{E}_{1}$, we can get the number of diagonal blocks in the $i$-th frontal slice of $ \mathcal{S}$ is $\mathit{R}_i$. From this, we give the vector ${\bf R} \in \mathbb{N_{+}}^{l}$ a definition in the following. 
\begin{definition} \label{definition 4.1}
	Let $\mathcal{A}\in \mathbb{C}^{m\times n \times l} $ be decomposed into $\mathcal{A} = \mathcal{U}  \ltimes  \mathcal{S} \ltimes  \mathcal{V}^{H}+\mathcal{E}_{1}$, and $\mathcal{S}\in \mathbb{C}^{ m\times n \times l}$ with each frontal slice a block-diagonal matrix. Then we call ${\bf R}=[R_{1}, R_{2}, \dots, R_{l}]^{T} \in \mathbb{N_{+}}^{l}$ the block rank, where $R_{i}$ represents the number of diagonal blocks in the $i$-th frontal slice of $\mathcal{S}$.
\end{definition}
\indent The MATLAB psuedocode of truncated STP-SVD is as follows.\\
\begin{algorithm}[H]
	\renewcommand{\algorithmicrequire}{\textbf{Input:}}
	\renewcommand{\algorithmicensure}{\textbf{Output:}}
	\caption{Truncated STP-SVD of Tensors}
	\begin{algorithmic}[1]     \label{algorithm 4}
		\REQUIRE $\mathcal{A}\in \mathbb{C}^{m_{1}m_{2}\times n_{1}n_{2}\times l} $, $m_{2},n_{2}$, ${\bf R}\in \mathbb{N_{+}}^{l}$.
		\ENSURE  $\mathcal{U}, \mathcal{S}, \mathcal{V}$.
		\STATE  Perform Fourier transform on $\mathcal{A}$, i.e.$\hat{\mathcal{A}}=\mathrm{fft}(\mathcal{A}, [\,], 3)$;
		\STATE From Lemma 4.1, compute $B_{i} $ and $C_{i} $ and do truncated SVD on $B_{i} $:\\
		$\mathbf{for}$ $i=1:l$ $\mathbf{do}$\\ 
		\quad $\hat{\mathcal{A}}(:,:,i) \approx B_{i} \otimes C_{i} $;\\
		\quad	$r=R(i)$;\\ 
		\quad $[ U_{i}, \mathrm{\Sigma}_{{B}_{i}}, V_{i}]=\mathrm{svds}(B_{i}, r)$;\\
		\quad  $\mathrm{\Sigma}_{i} =\mathrm{\Sigma}_{{B}_{i}}  \otimes C_{i} \in \mathbb{C}^{rm_{2}\times rn_{2}} $;\\
		\quad $\mathcal{U}_{i} \leftarrow U_{i}$, $\mathcal{S}_{i} \leftarrow \mathrm{\Sigma}_{i}$, $\mathcal{V}_{i} \leftarrow V_{i}$;\\
		$\mathbf{end}$ $\mathbf{for}$
		\STATE $\mathcal{U}=\mathrm{ifft}(\mathcal{U}, [\,], 3)$, $\mathcal{S}=\mathrm{ifft}(\mathcal{S}, [\,], 3)$, $\mathcal{V}=\mathrm{ifft}(\mathcal{V}, [\,], 3)$.\\
		($\hat{\mathcal{A}}(:,:,i)$, $\mathcal{U}_{i}$, $\mathcal{S}_{i}$, and $\mathcal{V}_{i}$ denote the $i$-th frontal slice of $\hat{\mathcal{A}}$, $\mathcal{U}$, $\mathcal{S}$, and $\mathcal{V}$, respectively.) \\	   			   					
	\end{algorithmic}
\end{algorithm}
\indent Now, we give a specific error analysis for the truncated STP-SVD of tensors. First, with the same assumptions and symbols in Theorem \ref{theorem 4.2}, we give tensors truncated STP-SVD a similar representation to (\ref{equation 4.10}), expressed by 
\begin{equation} \label{4.14}
\begin{aligned}
\mathcal{A} = \mathcal{U}  \ltimes  \mathcal{S} \ltimes  \mathcal{V}^{H}+\mathcal{E},
\end{aligned}
\end{equation} 
where $\mathcal{E}$ is the
corresponding error tensor. Next, we have the upper bound for  $\mathcal{E}$, \\
\begin{equation*}
\begin{aligned}
\Vert \mathcal{E}  \Vert_{F}  &= \sqrt{ {\Vert E^{(1)}\Vert}^{2}_{F}+  {\Vert E^{(2)}\Vert}^{2}_{F}+\dots +  {\Vert E^{(l)}\Vert}^{2}_{F}}\\
 &\leqslant \Vert E^{(1)}\Vert_{F} +\Vert E^{(2)}\Vert_{F}+ \cdots +\Vert E^{(l)}\Vert_{F},
 \end{aligned}
\end{equation*}
where $\Vert E^{(k)}\Vert_{F}$ $(k=1,2, \cdots, l)$ denotes the upper bound of error produced by using truncated SVD on $\hat{\mathcal{A}}(:,:,k)$. This is changed to consider the error upper bound of matrix STP-SVD (see (\ref{4.10})). Each  $\Vert E^{(k)}\Vert_{F} $ can be given as 
\begin{equation*}
\Vert E^{(k)}\Vert_{F}  \leqslant  \sqrt{\sum_{i=2}^{q}(\tilde{\sigma}_{i}^{(k)})^{2}}+\sqrt{\sum_{j=R_{k}+1}^{p}{\Vert S_{j}^{(k)}\Vert}^{2}_{F}},
\end{equation*}
from (\ref{4.10}). Suppose $\tilde{\mathcal{A}}(:,:,k)$ is obtained by (\ref{equation 4.1}) with $A=\hat{\mathcal{A}}(:,:,k)$ therein, and $\tilde{\sigma}_{i}^{(k)}$ with $i=1, 2, \cdots, q$, $q=min\left\lbrace m_{1}n_{1},m_{2}n_{2}\right\rbrace$ are singular values of $\tilde{\mathcal{A}}(:,:,k)$. The matrix $S_{j}^{(k)}$ denotes the error produced by using truncated SVD on $\hat{\mathcal{A}}(:,:,k)$, and $p=min\left\lbrace m_{1},n_{1}\right\rbrace$. Then an upper bound for approximation error $\mathcal{E}  $ in (\ref{4.14}) is
\begin{equation*}
\Vert \mathcal{E}  \Vert_{F}  \leqslant  \sum_{k=1}^{l} \left( \sqrt{\sum_{i=2}^{q}(\tilde{\sigma}_{i}^{(k)})^{2}}+\sqrt{\sum_{j=R_{k}+1}^{p}{\Vert S_{j}^{(k)}\Vert}^{2}_{F}}\right).	   
\end{equation*}
\subsection{Data Compression of New Tensor Decomposition }
\indent The algorithms introduced in this paper can be used for data compression. In fact, the dominant cost for our algorithm is the STP-SVD for $\hat{\mathcal{A}}(:,:,i)$ (the $i$-th frontal slice after the Fourier transform of $\mathcal{A}$), therefore, when calculating the required storage, we mainly consider STP-SVD for $\hat{\mathcal{A}}(:,:,i)$. For example, we suppose  $\mathcal{A}\in \mathbb{C}^{m\times n \times l} $ with $m=m_{1}m_{2}$ and $n=n_{1}n_{2}$ is a third-order tensor. Then, we need to store $m\times n \times l$ data for storing $\mathcal{A}$ in the computer. If we use full T-SVD \cite{kilmer2011factorization,t-product}, we need to store $(m+n+1)pl$, $p=min\left\lbrace m,n\right\rbrace $ data. However, if we use truncated SVD for each $\hat{\mathcal{A}}(:,:,i)$ in T-SVD algorithm, and truncated at $r$ every time, then we need to store $(m+n+1)rl$ data. If we choose $r\ll p$, the truncated T-SVD will achieve data compression to a certain extent. Now, if we approximate $\mathcal{A}$ by using STP-SVD, the amount of required storage drops to $\left[ (m_{1}+n_{1}+1)q+m_{2}n_{2}\right]l$, $q=min\left\lbrace m_{1},n_{1}\right\rbrace $. And if we use truncated STP-SVD, and truncated at $r$ in the same way, then the data we need to store is only $\left[ (m_{1}+n_{1}+1)r+m_{2}n_{2}\right]l$. Compared to T-SVD, our algorithm requires less data when storing third-order tensors, data compression is well implemented therefore. We can see the result specifically in Table \ref{table 1}.
\begin{table}[H]  
	\centering 
		\caption{The algorithms and corresponding required storage for decomposing an $m\times n \times l$ tensor $\mathcal{A}$.} 
	\setlength{\tabcolsep}{20pt} 	     	
	\scalebox{0.9}{	
		\renewcommand\arraystretch{0.8} 
		\begin{tabular}{ lc } 
			\toprule \\	
			\textbf{Algorithm}& \textbf{Required storage} \\
			\midrule \\
			full T-SVD& $(m+n+1)pl$, $p=min\left\lbrace m,n\right\rbrace $  \\
			\midrule  \\
			full STP-SVD	& $\left[ (m_{1}+n_{1}+1)q+m_{2}n_{2}\right]l$, $q=min\left\lbrace m_{1},n_{1}\right\rbrace $ \\
			\midrule  \\
			truncated T-SVD  & $(m+n+1)rl$   \\  
			\midrule  \\
			truncated STP-SVD& $\left[ (m_{1}+n_{1}+1)r+m_{2}n_{2}\right]l$  \\ 
			\bottomrule
			Note: $m=m_1m_2$ and $n=n_1n_2.$
		\end{tabular}}
		\label{table 1} 
	\end{table}
	\indent Next we introduce the data compression rate   \\
	\begin{equation} \label{equation 4.12} 
	Cr=\frac{N}{N_{O}},
	\end{equation}
	where $N$ denotes the amount of data needs to storage when we use strategy to compress the $m\times n\times l$ tensor $\mathcal{A}$, and ${N_{O}}$ denotes the original tensor without compression. Obviously, $Cr<1$ represents compression strategy stores less data compared to storing the entire tensor directly, while $Cr>1$, the reverse applies. For any $m\times n\times l$ third-order tensor with $m=m_{1}m_{2}$ and $n=n_{1}n_{2}$, the compression rate  be showed in the following Table \ref{table 2}.	\\
	\begin{table}[h]  
		\centering  
				\caption{Data compression rate for decomposing $\mathcal{A}\in \mathbb{C}^{m\times n \times l} $.}
		\setlength{\tabcolsep}{20pt}	     	
		\scalebox{0.9}{	
			\renewcommand\arraystretch{0.8}   	     	  		
			\begin{tabular}{ lc } 
				\toprule \\	
				\textbf{Algorithm}& \textbf{Cr} \\
				\midrule \\
				full T-SVD& $\dfrac{(m+n+1)p}{mn}$, $p=min\left\lbrace m,n\right\rbrace $   \\
				\midrule  \\
				full STP-SVD	& $\dfrac{(m_{1}+n_{1}+1)q+m_{2}n_{2}}{mn}$, $q=min\left\lbrace m_{1},n_{1}\right\rbrace $ \\
				\midrule  \\
				truncated T-SVD  & $\dfrac{(m+n+1)r}{mn}$ \\  
				\midrule  \\
				truncated STP-SVD&$\dfrac{(m_{1}+n_{1}+1)r+m_{2}n_{2}}{mn}$ \\ 
				\bottomrule
			\end{tabular}}
			\label{table 2}
		\end{table}
		\indent Now, we assume that $\mathcal{A}$ is an $n \times n \times n$ third-order tensor, and taking $m_{1}$=$m_{2}$=$n_{1}$=$n_{2}$=$\sqrt{n}$ in our algorithms. Then, the number of data we need to store is $n^{3}$ for saving $\mathcal{A}$. If we use full  T-SVD \cite{kilmer2011factorization,t-product} for $\mathcal{A}$, we need to store $ 2n^{3}+n^{2 }$ data. For truncated T-SVD, if we truncate at $r$ every time we decompose $\hat{\mathcal{A}}(:,:,i)$, then we need to store $2rn^{2}+rn$ data. If we use Algorithm \ref{algorithm 3} to approximate $\mathcal{A}$, the required storage capacity is $3n^{2}+n^{\frac{3}{2}}$. If we use truncated STP-SVD (Algorithm \ref{algorithm 4}) on $\mathcal{A}$, and truncated at $r$ in the same way, then the data we need to store is $n^{2}+2n^{\frac{3}{2}}r+nr$. We can also give the corresponding data compression rate of  $\mathcal{A} \in \mathbb{C}^{n \times n \times n}$ on the basis of (\ref{equation 4.12}). The result is listed in Table \ref{table 3}.
		\begin{table}[h]  
			\centering  
			\caption{The corresponding required storage and compression rate of algorithms for decomposing an $n \times n \times n$ tensor $\mathcal{A}$.}
			\setlength{\tabcolsep}{20pt}	    	  	    	
			\scalebox{0.9}{	
				\renewcommand\arraystretch{0.8}   	     		 			  		
				\begin{tabular}{ lcc } 
					\toprule \\	
					\textbf{Algorithm}& \textbf{Required storage}&\textbf{Cr}\\
					\midrule \\
					full T-SVD& $2n^{3}+n^{2 } $& $\dfrac{2n+1}{n}$  \\
					\midrule  \\
					full STP-SVD	& $3n^{2}+\sqrt{n^{3}} $&$\dfrac{3\sqrt{n}+1}{\sqrt{n^{3}}}$  \\
					\midrule  \\
					truncated T-SVD  & $2rn^{2}+rn$ &$\dfrac{2rn+r}{n^{2}}$ \\  
					\midrule  \\
					truncated STP-SVD& $n^{2}+2r\sqrt{n^{3}}+nr$ &$\dfrac{n+2r\sqrt{n}+r}{n^{2}}$\\ 
					\bottomrule
				\end{tabular}}
				\label{table 3}
			\end{table}

\section{Applications}\label{sec:5}  
		\indent One of the most important applications of our theoretical knowledge is high-resolution color image compression. Therefore, in this section, we give some numerical experiments related to image compression.\\	
			\indent We use peak signal-to-noise ratio (PSNR) \cite{huynh2008scope} and structural similarity (SSIM) \cite{wang2004image} to measure the quality of image compression here. Often, after image compression, the output image will differ to some extent from the original image. In order to measure the quality of the processed image, it is common to refer to the PSNR value to determine whether a particular process is satisfactory. And the larger the PSNR value, the better the image quality. The value of PSNR is 30-40dB usually indicates that the image quality is good. SSIM is an indicator to measure the similarity of two images, and its value range is [0, 1]. Of the two images used by SSIM, one is the original uncompressed image and the other is the distorted image after processing. The larger the value of SSIM, the smaller the degree of image distortion.\\
	\textbf{	Experiment 1}	 For a $4000\times 6000 \times 3$ color image, we choose the input factors $m_{2}=4$, $n_{2}=6$, then use the algorithm introduced in section 4 to compress it. We give a comparison of compression quality which obtained by using truncated T-SVD with $R=[200, 200, 200]^{T}$ and STP-SVD without truncation, respectively, and show the results in Table \ref{table 4}. \\
	\begin{figure*}[h]
		\begin{subfigure}[t]{0.32\textwidth}
		\centering		
		\includegraphics[width=\textwidth]{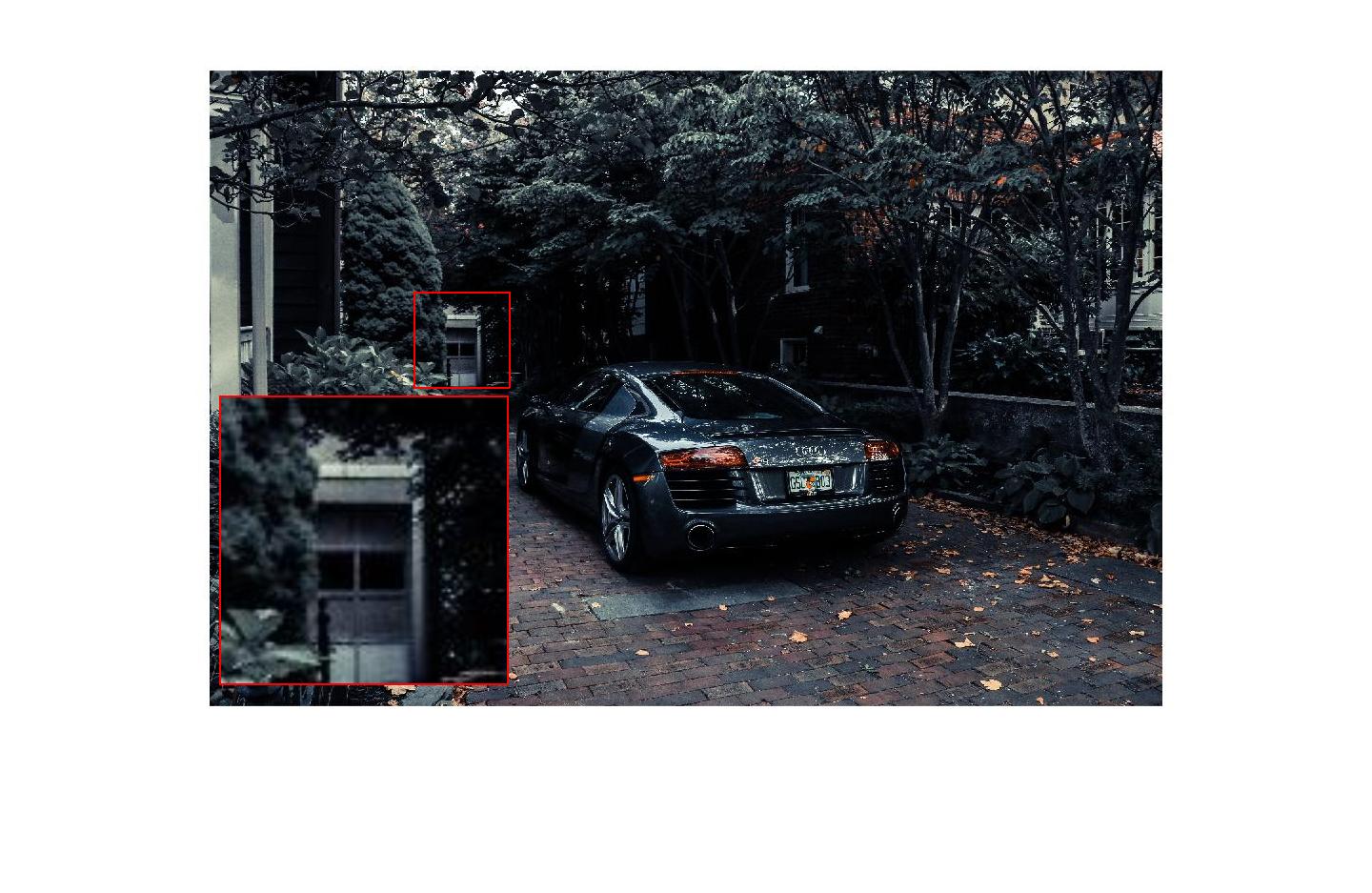}	\caption{Original image}
	 \end{subfigure}
		\begin{subfigure}[t]{0.32\textwidth}
			\centering		
		\includegraphics[width=\textwidth]{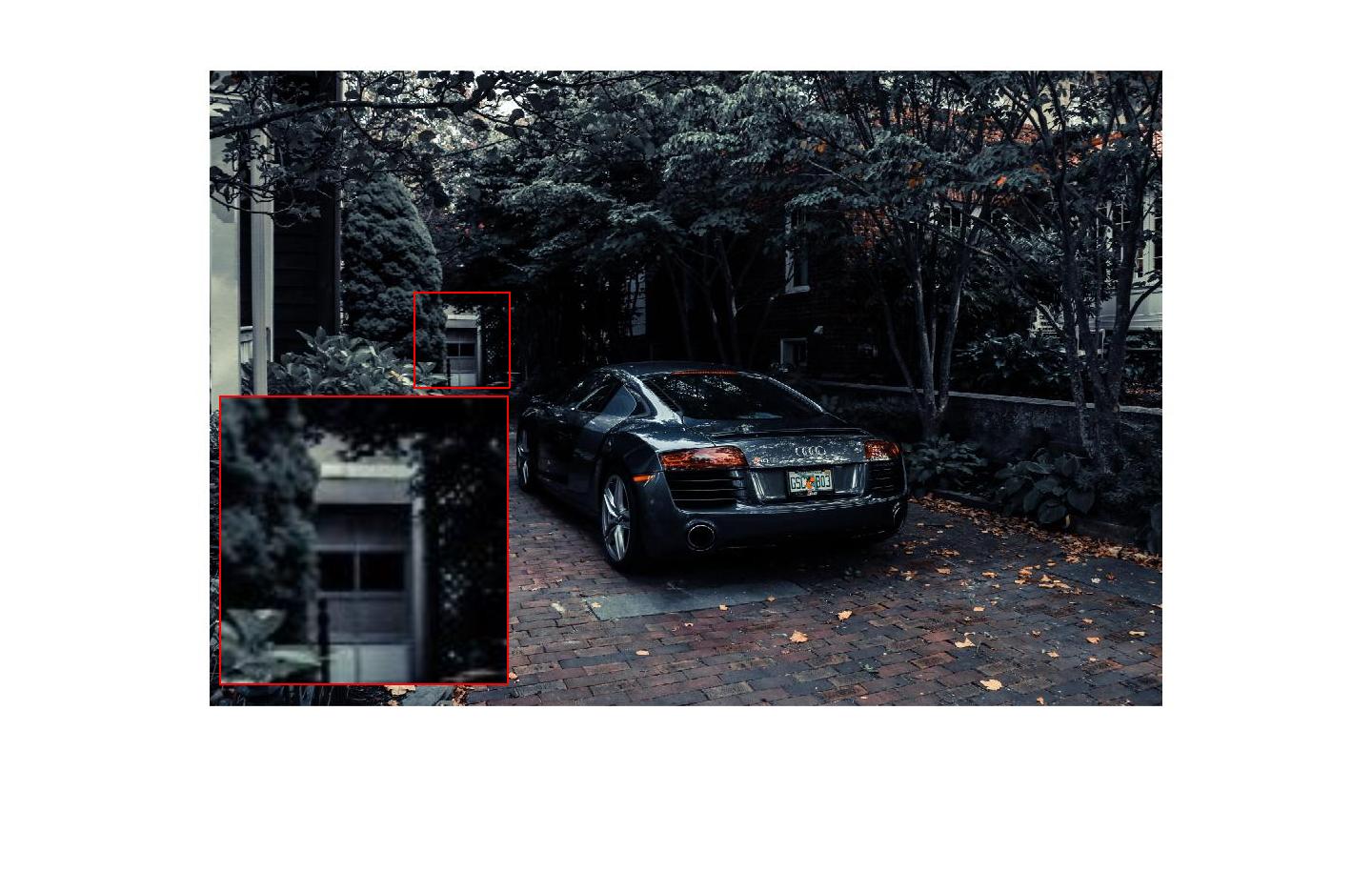}		\caption{STP-SVD without truncation}
	 \end{subfigure}
	 	\begin{subfigure}[t]{0.32\textwidth}
			\centering		
		\includegraphics[width=\textwidth]{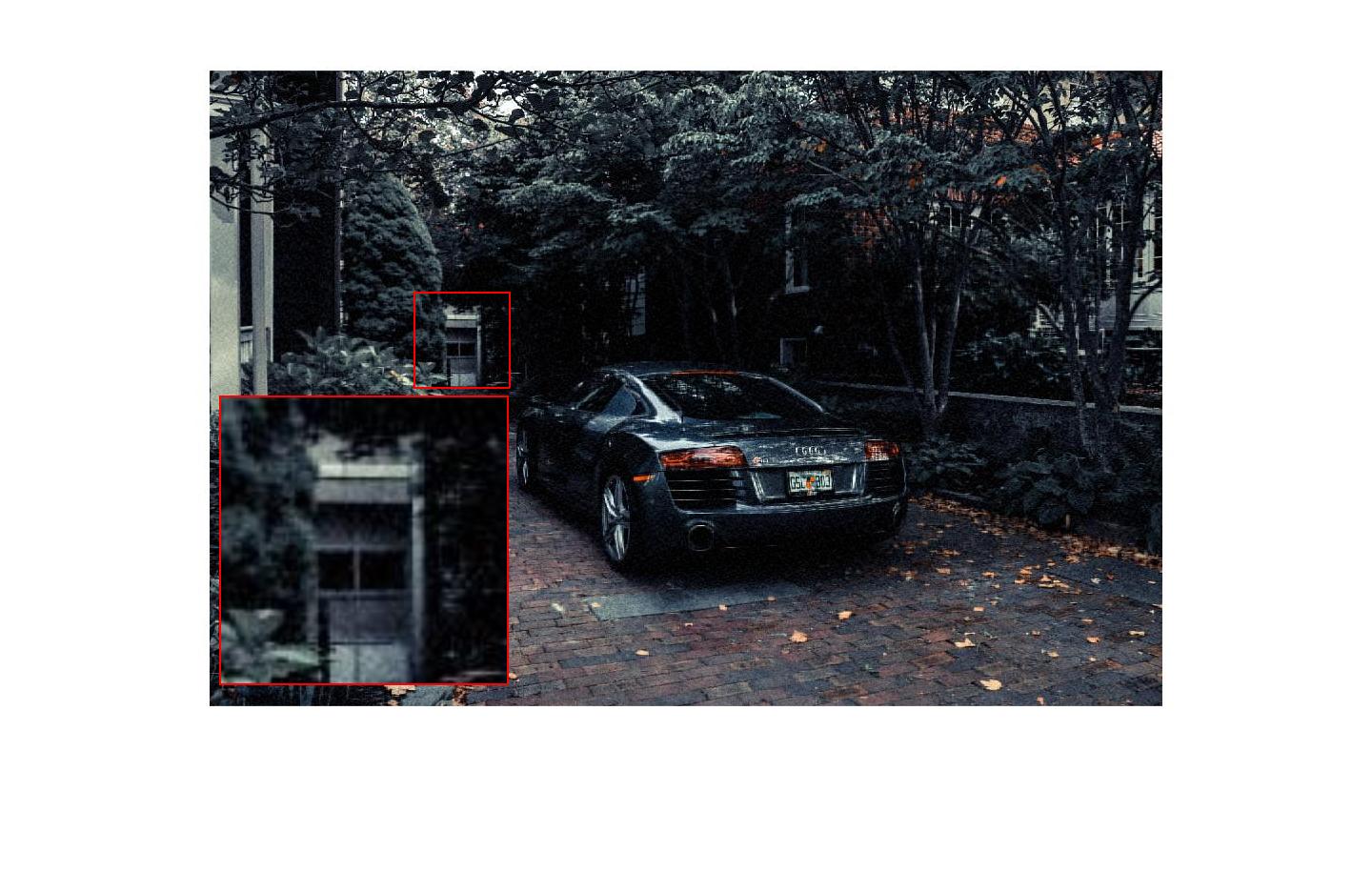}		\caption{Truncated T-SVD}
	 \end{subfigure}
		\caption{The original image and compressed images obtained by using STP-SVD and truncated T-SVD.}
		\end{figure*}
				\begin{table}[h] 
					\centering 	
					\caption{Comparison of compression quality which obtained by using truncated T-SVD with $R=[200, 200, 200]^{T}$ and STP-SVD without truncation for Image 01.}\setlength{\tabcolsep}{18.7pt}	     	
						\scalebox{0.85}{	
							\renewcommand\arraystretch{0.8}  			
						\begin{tabular}{lcc}
							\toprule \\							
							 	& \textbf{STP-SVD} & \textbf{Truncated T-SVD} \\
							 	& \textbf{without truncation}     & \textbf{with}  $R={[200, 200, 200]^{\mathrm{T}}}$   \\
							\midrule  \\ 
							\textbf{TIME} & 15.888047s & 119.671456s \\
							\midrule  \\ 
							\textbf{Related Error} & 0.2117 & 0.2169\\
							\midrule  \\ 
							\textbf{PSNR} & 25.3376 & 25.2777    \\
							\midrule  \\ 
							 \textbf{SSIM}& 0.7981 & 0.7026\\							 		
							\bottomrule
						\end{tabular}}
						\label{table 4} 
					\end{table}  
			\indent It's easy to see that STP-SVD without truncation has almost the same or even better experimental results compared with truncated T-SVD with $R=[200, 200, 200]^{T}$. However, STP-SVD without truncation saves a lot of time. \\ 	\\	
	\noindent	\textbf{	Experiment 2} In Table \ref{table 5}, we compare the compression quality which obtained by using truncated T-SVD and truncated STP-SVD for several images with different resolutions. Here we emphasize $R$ represents different meanings in truncated T-SVD and truncated STP-SVD. For T-SVD, $R_{i}$ represents the rank taken when SVD is performed on the $i$-th frontal slice of the target tensor, while the meaning of $R$ in STP-SVD is the same as that explained in Definition \ref{definition 4.1}. The image scales used for numerical experiments are 4887 $\times$ 7500 $\times$ 3, 3000 $\times$ 3000 $\times$ 3, 6000 $\times$ 8000 $\times$ 3, 4000 $\times$ 6000 $\times$ 3, and 5304 $\times$ 7952 $\times$ 3, the images results and data results of the numerical experiments are shown below. The first column((a), (d), (g), (j), (m)) of Figure \ref{figure6} shows the original images, the second column((b), (e), (h), (k), (n)) shows the truncated STP-SVD processed images, and the third column((c), (f), (i), (l), (o)) shows the truncated T-SVD processed images.\\ 
	
		\begin{figure}[htbp]
  \centering
  \subfloat[]{
    \includegraphics[width=5cm]{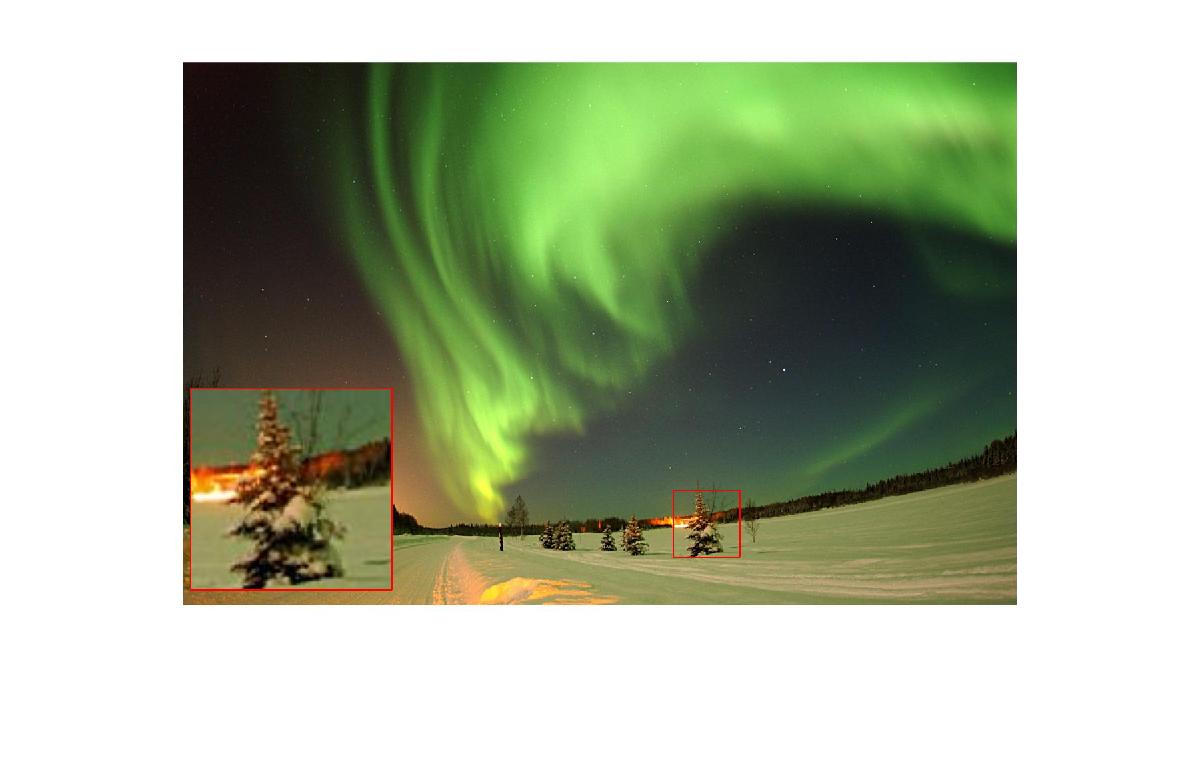}
  }
  \subfloat[]
  {
    \includegraphics[width=5cm]{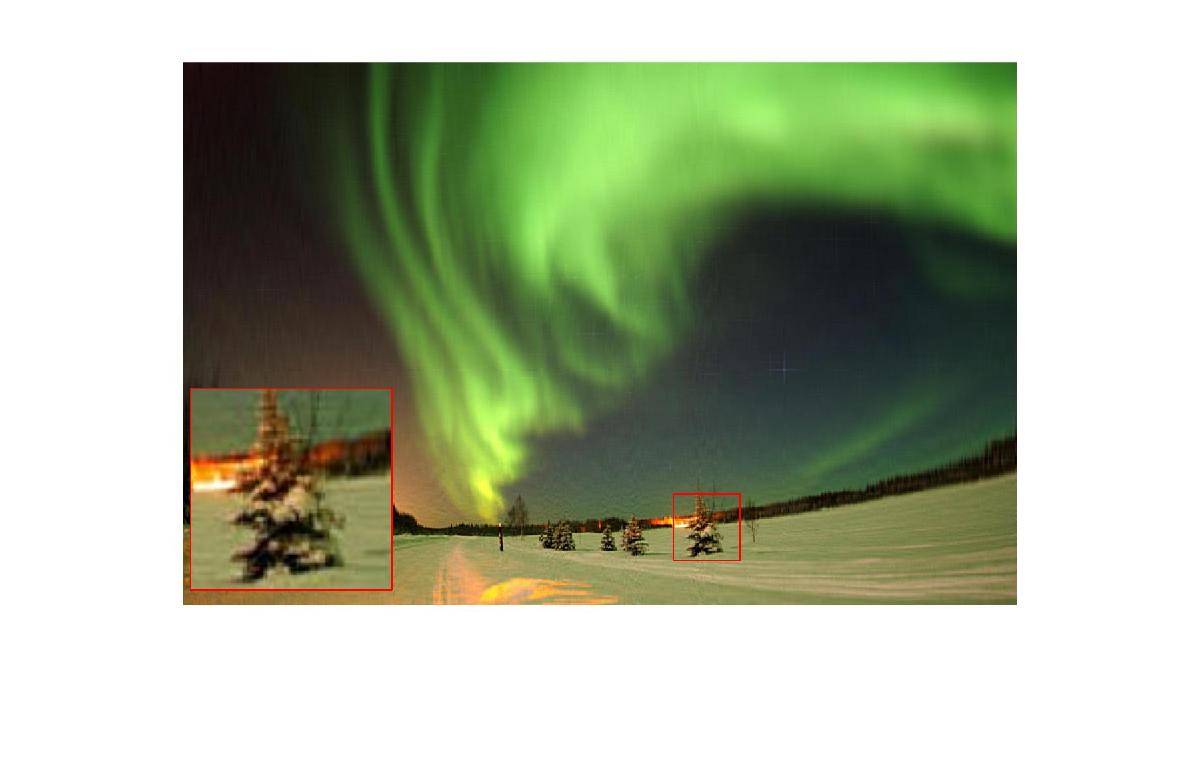}
  }
  \subfloat[]
  {
    \includegraphics[width=5cm]{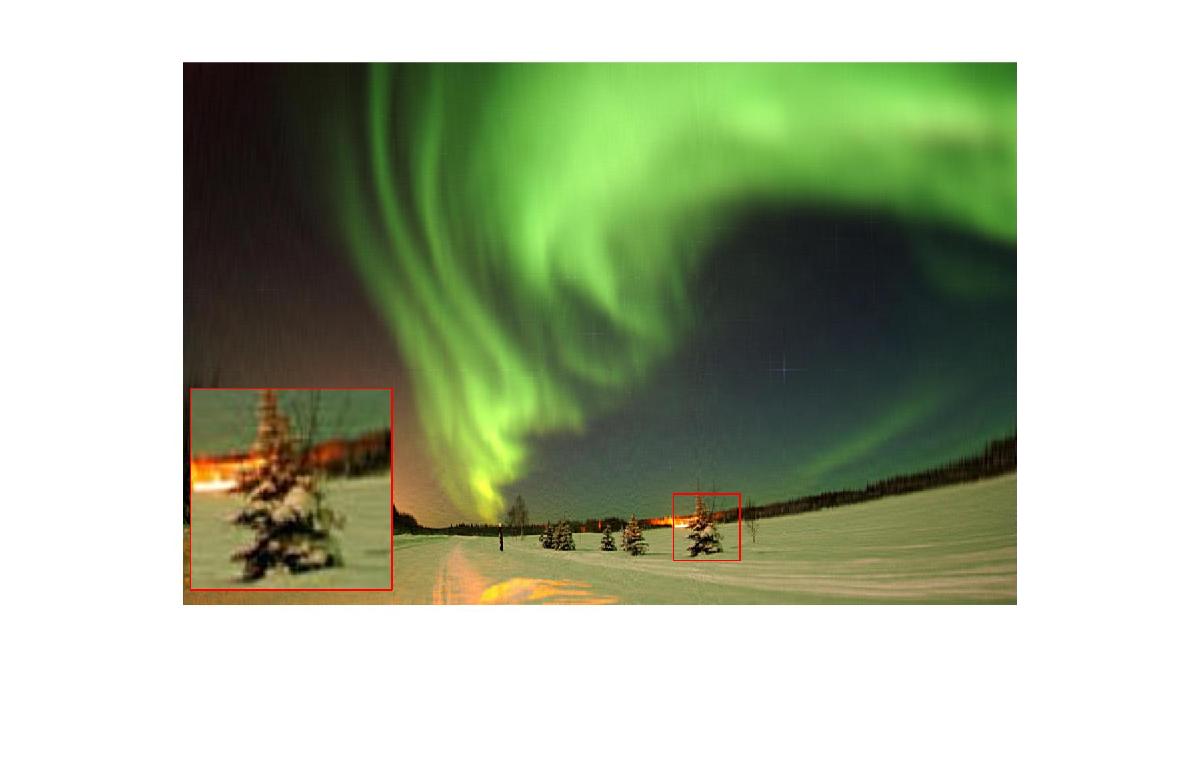}
  }
  
  \subfloat[]{
    \includegraphics[width=5cm]{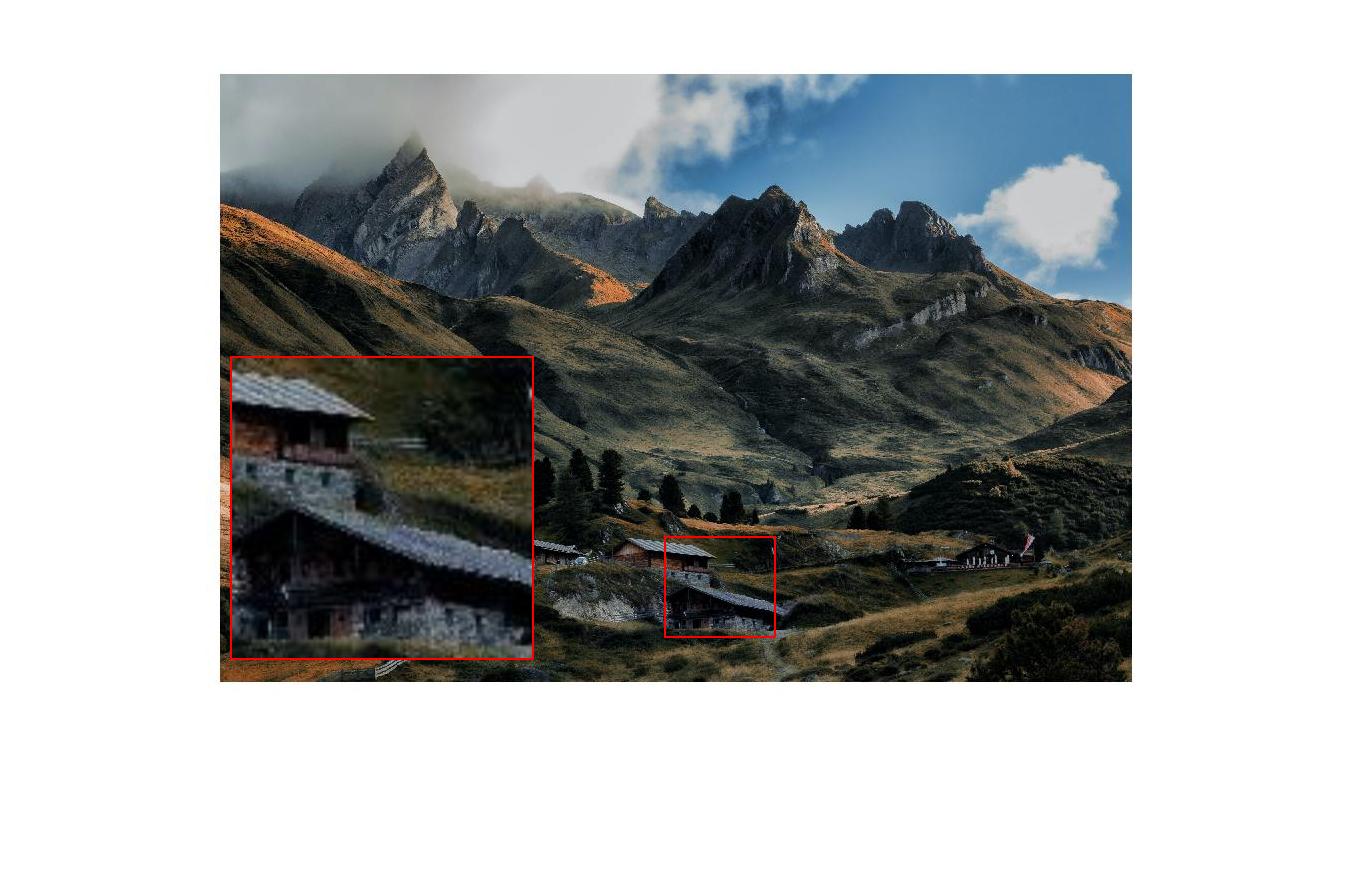}
  }
  \subfloat[]
  {
    \includegraphics[width=5cm]{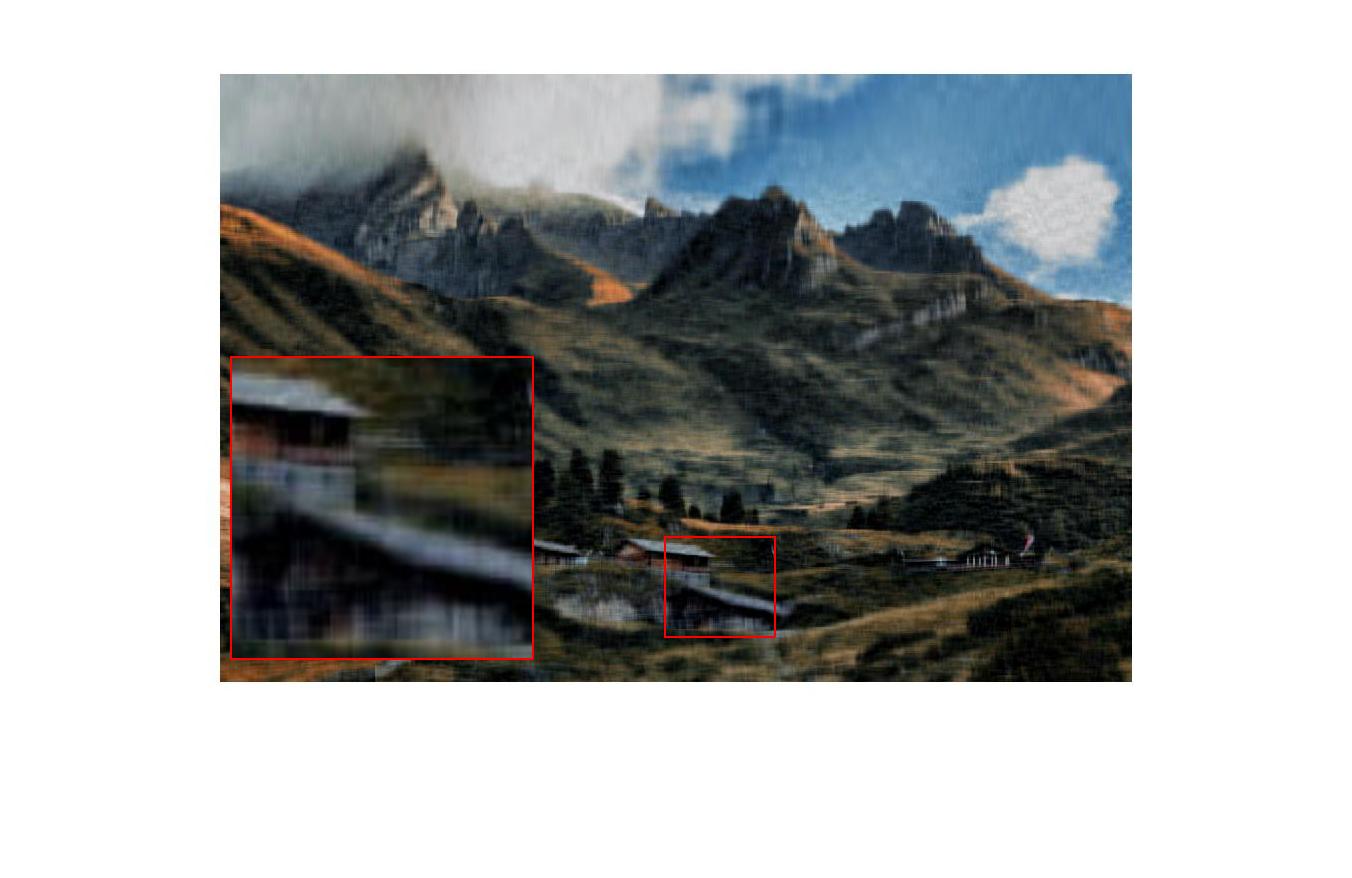}
  }
  \subfloat[]
  {
    \includegraphics[width=5cm]{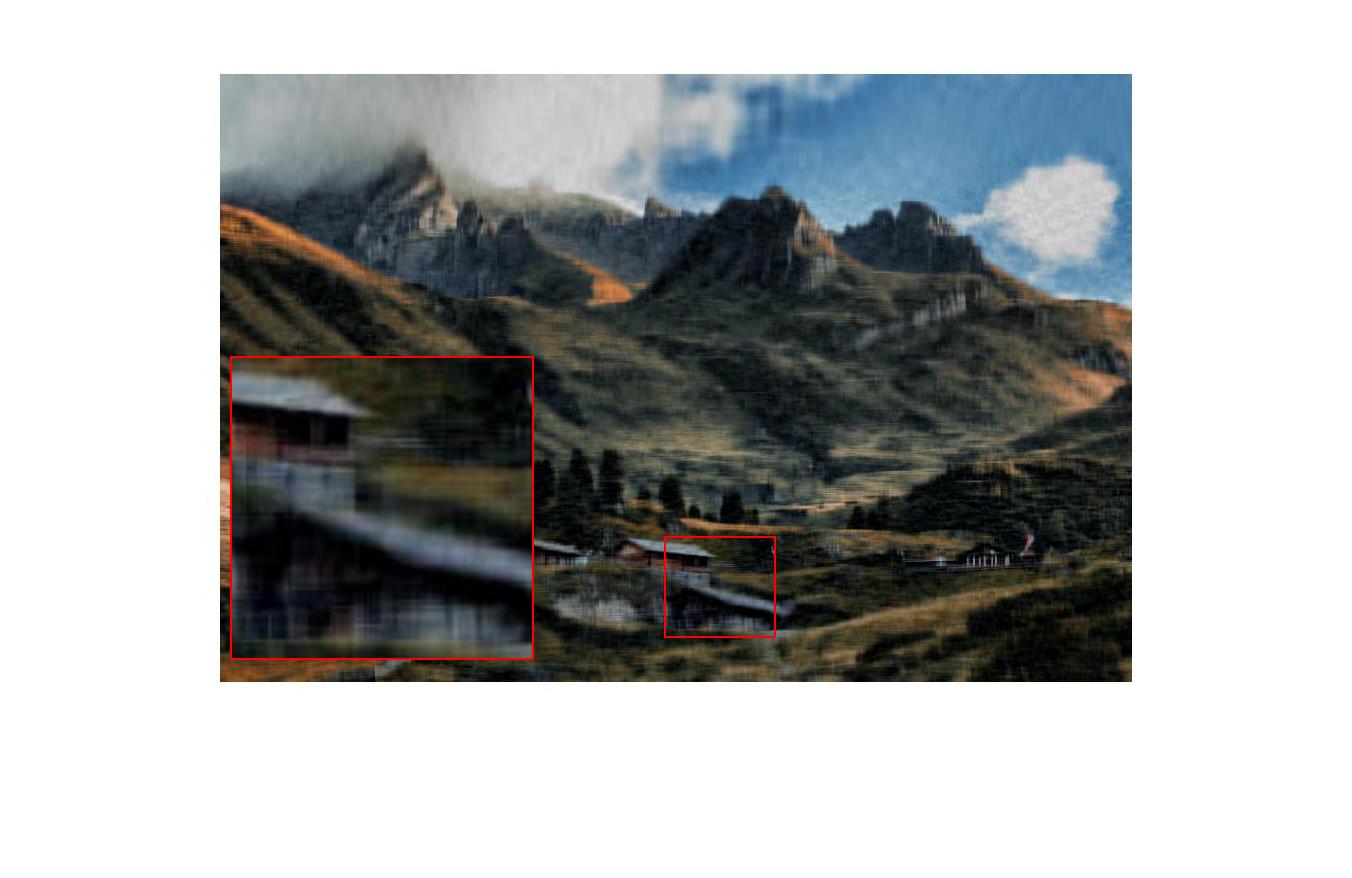}
  }
  
  \subfloat[]
  {
    \includegraphics[width=5cm]{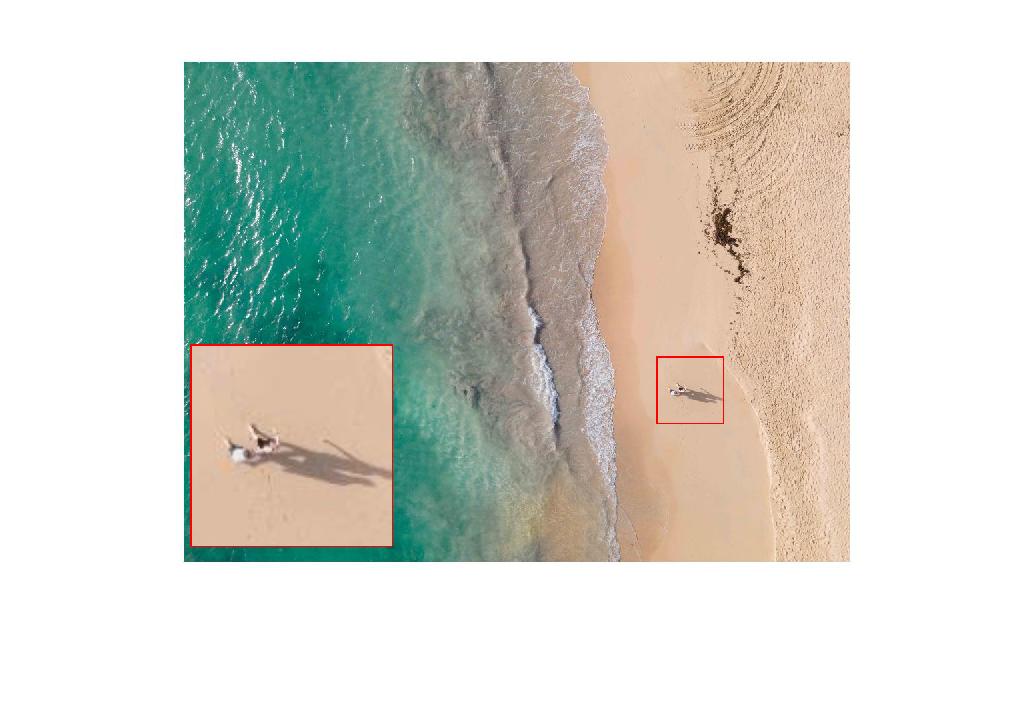}
  }
  \subfloat[]
  {
    \includegraphics[width=5cm]{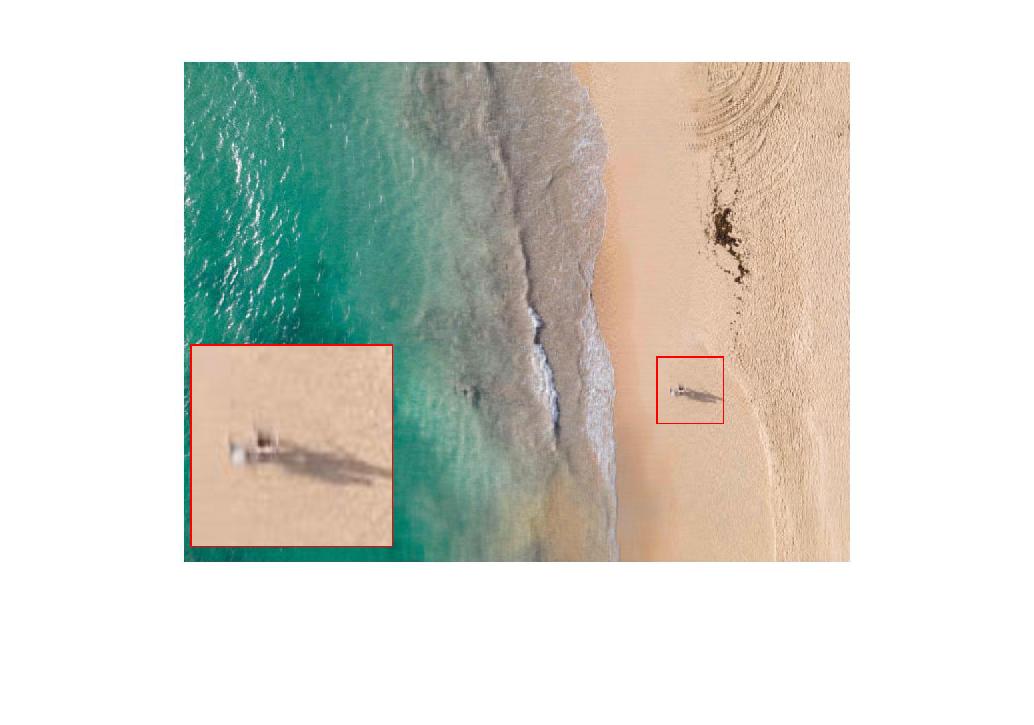}
  }
  \subfloat[]
  {
    \includegraphics[width=5cm]{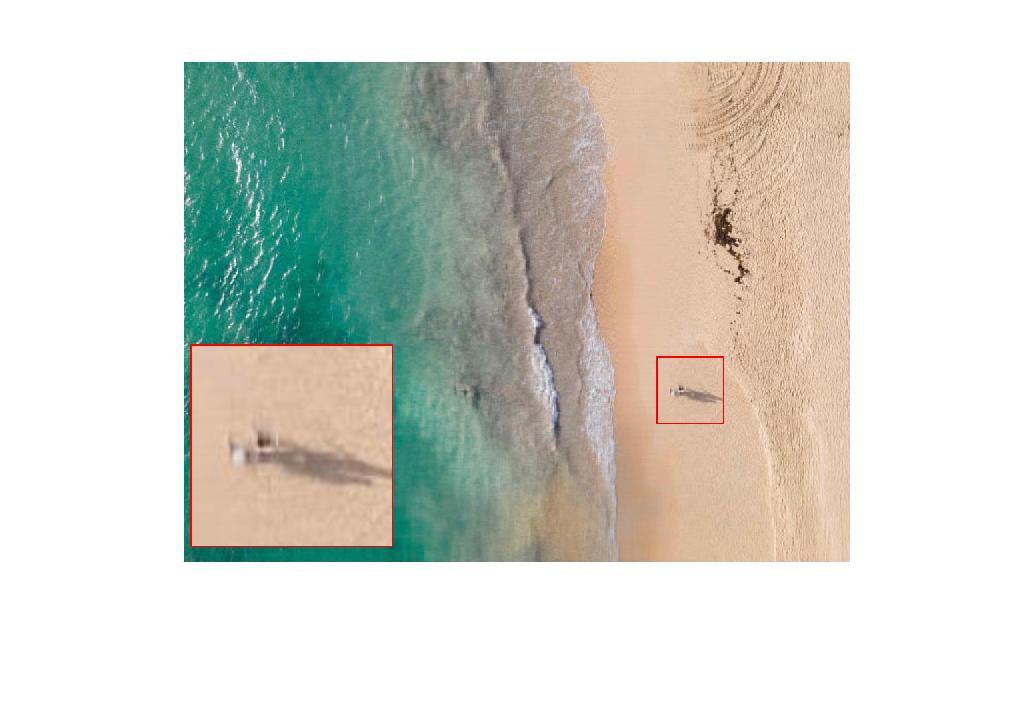}
  }
  
  \subfloat[]
  {
    \includegraphics[width=5cm]{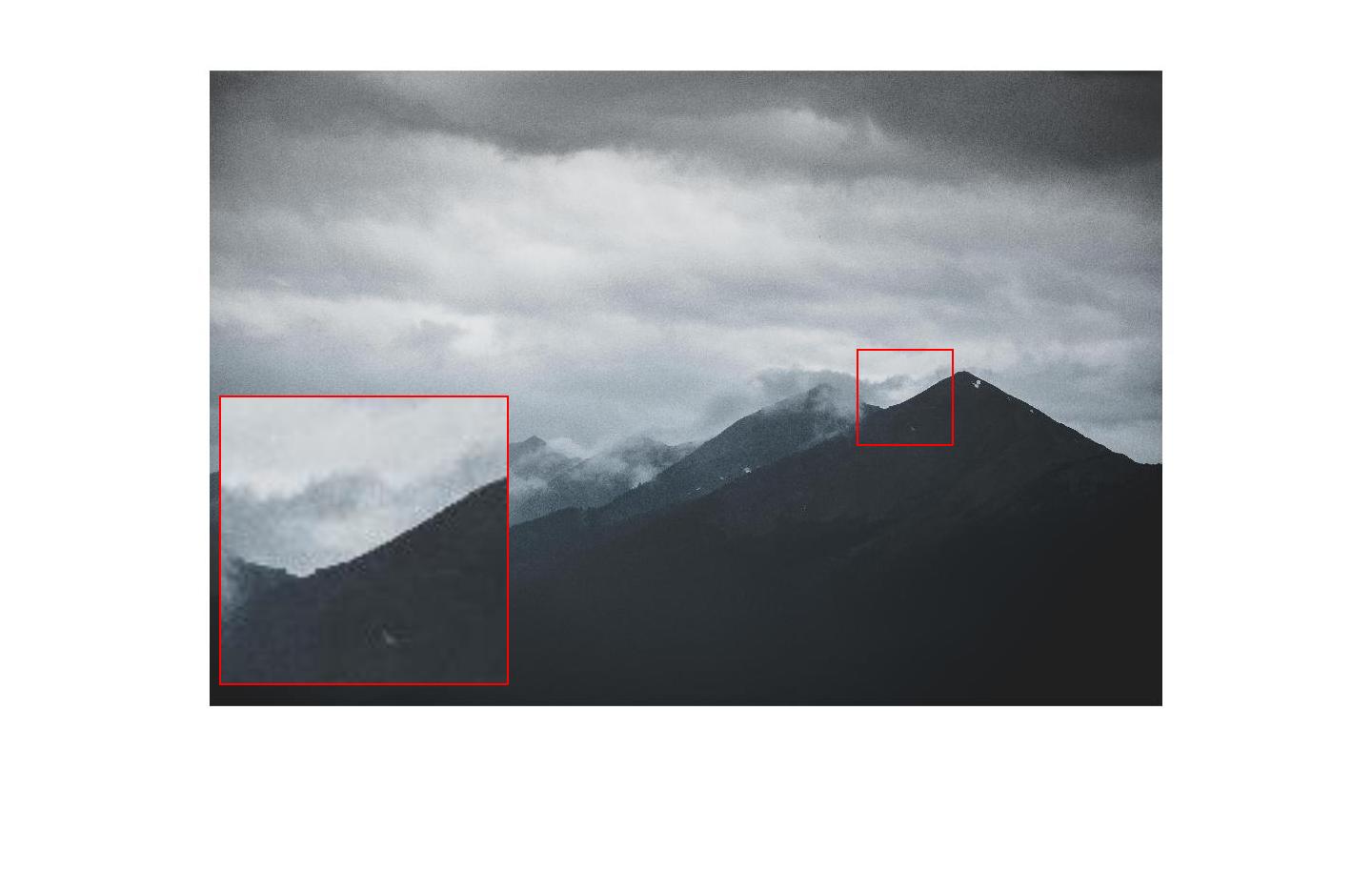}
  }
  \subfloat[]
  {
    \includegraphics[width=5cm]{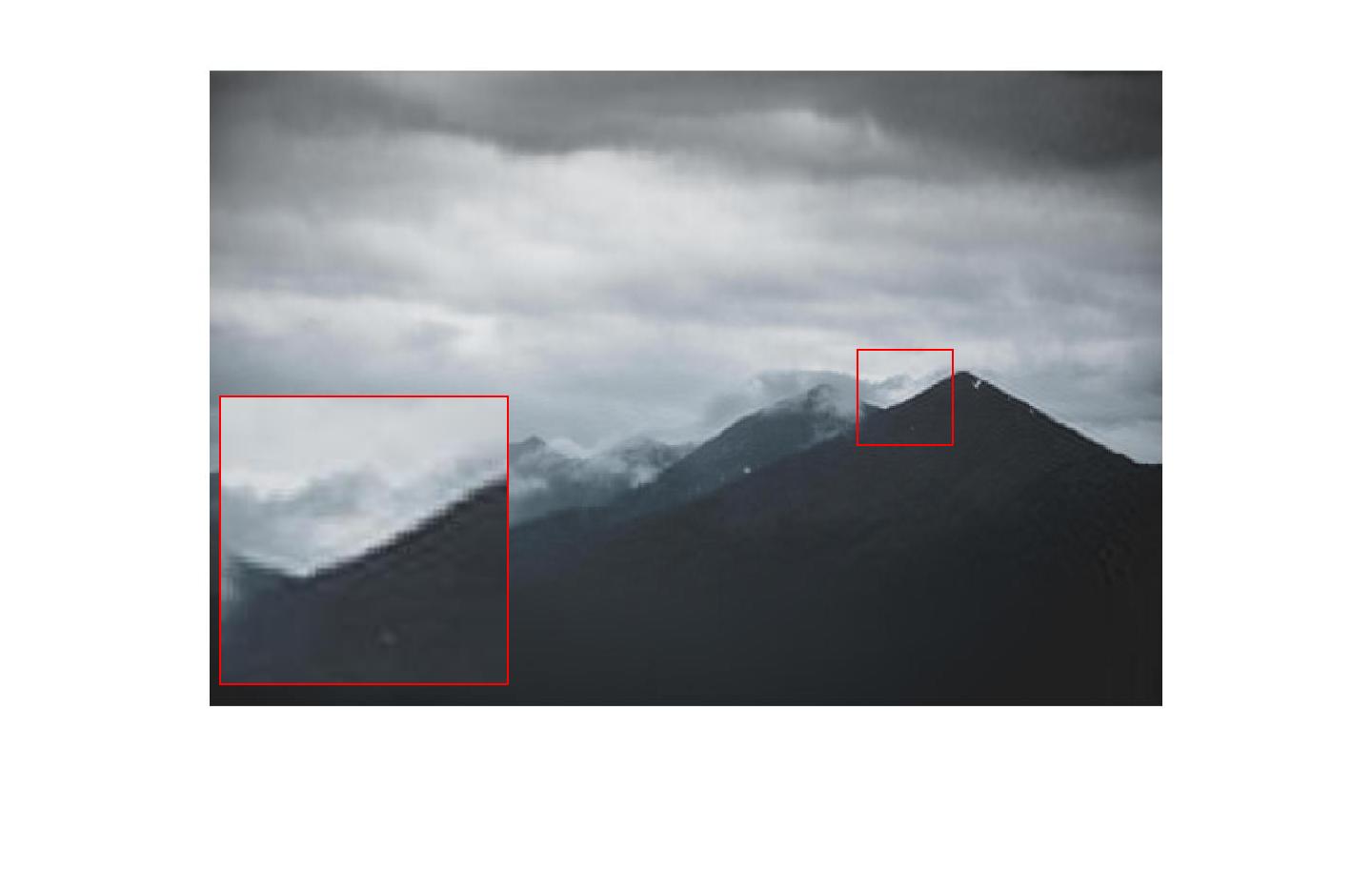}
  }
  \subfloat[]
  {
    \includegraphics[width=5cm]{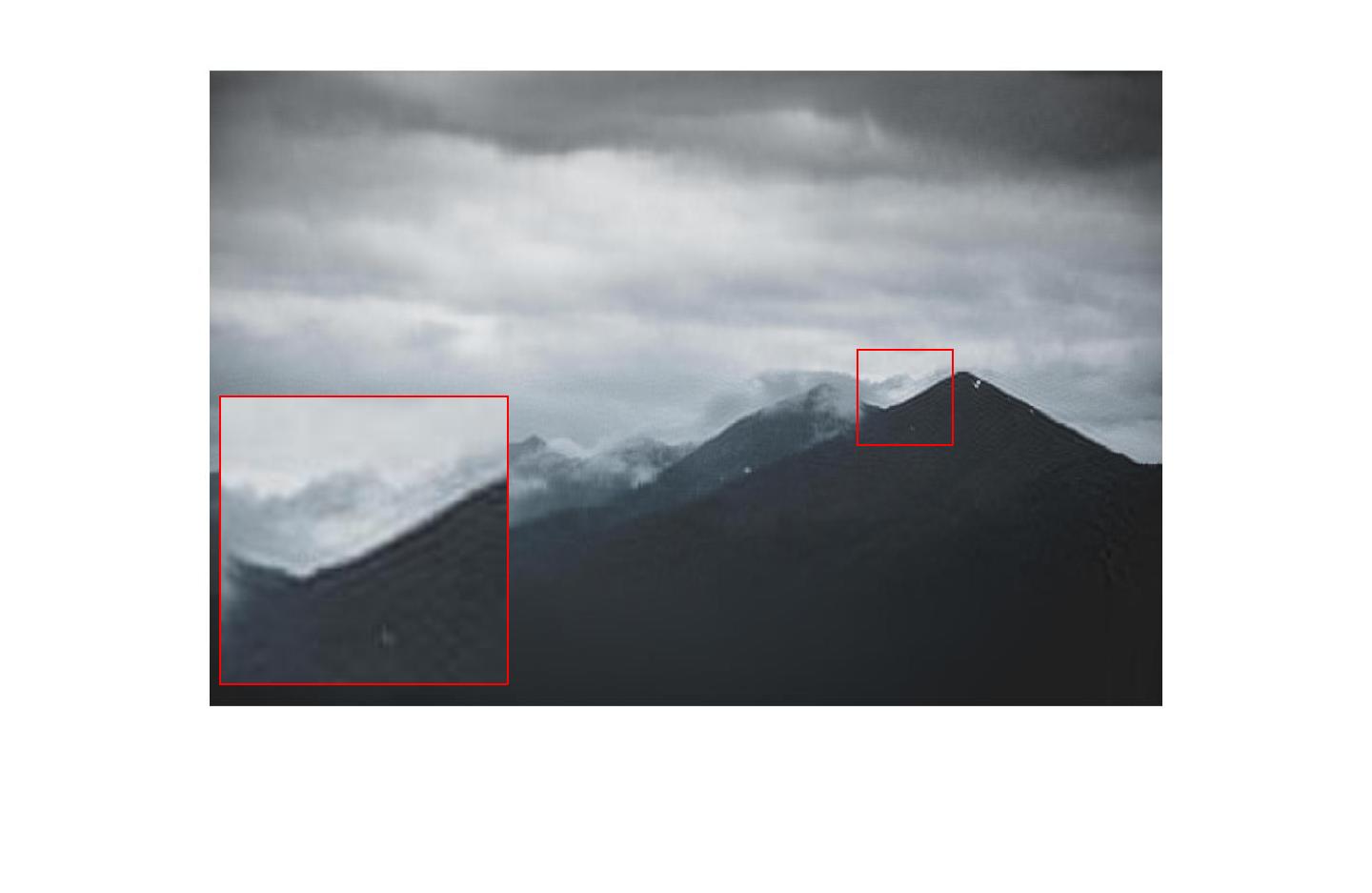}
  }
  
  \subfloat[]
  {
    \includegraphics[width=5cm]{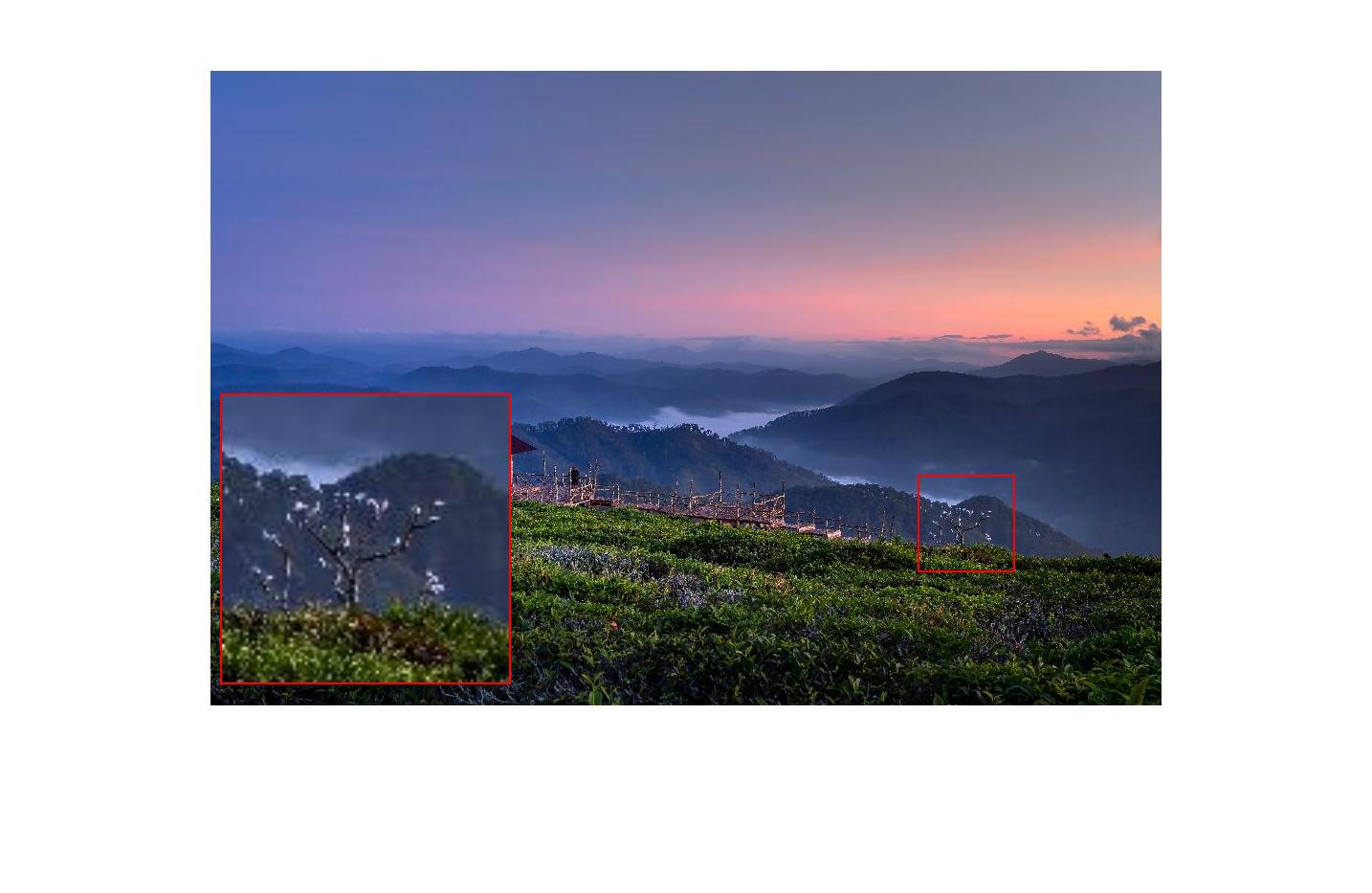}
  }
  \subfloat[]
  {
    \includegraphics[width=5cm]{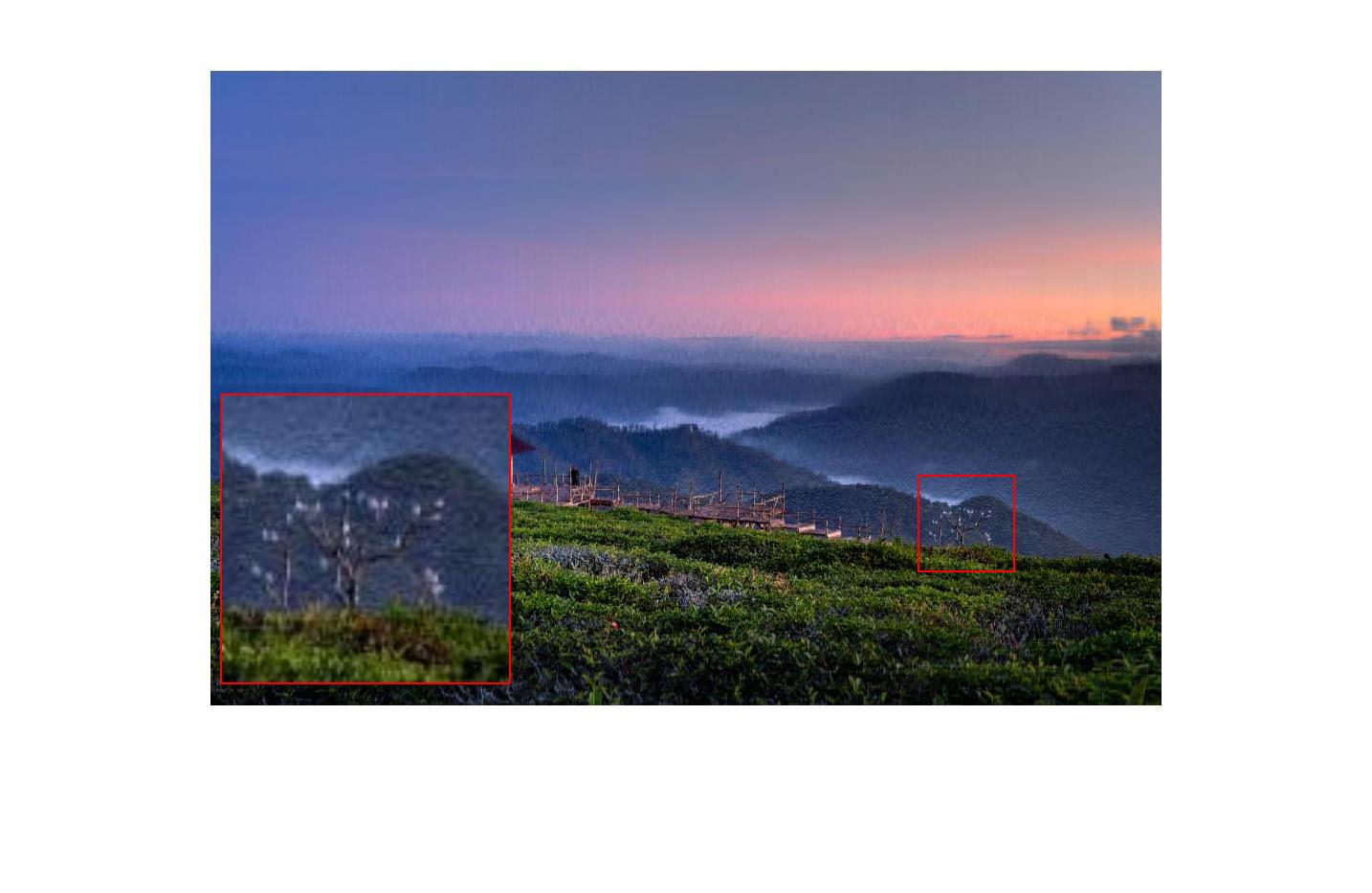}
  }
  \subfloat[]
  {
    \includegraphics[width=5cm]{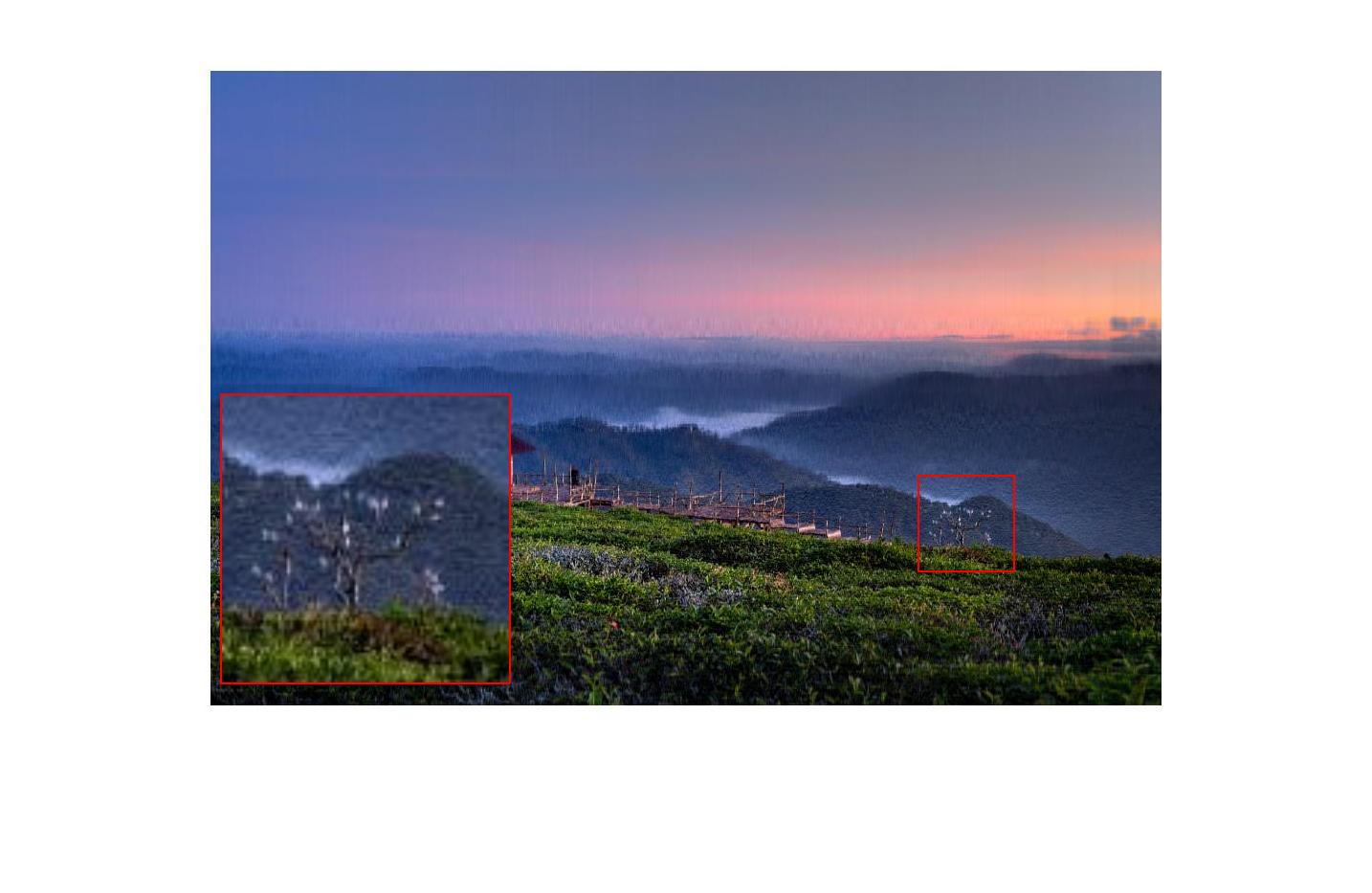}
  }
  \caption{The result images of numerical experiments. }
  \label{figure6}
\end{figure}
					
		\begin{table}[htbp] 
			\centering 	
			\caption{ The data results of numerical experiments.}
			\setlength{\tabcolsep}{5pt}	     	
			\scalebox{0.8}{	
				\renewcommand\arraystretch{1.5}  			
				\begin{tabular}{cccccc}
					\toprule \\							
			\textbf{Image Scale}	&$\mathbf{m_{2}}, \mathbf{n_{2}}$&  $\boldsymbol{R}$& \textbf{Index}& \textbf{TSTP-SVD} & \textbf{TT-SVD} \\			
					\midrule  \\ 
				\multirow{4}{*}{4887 $\times$ 7500 $\times$ 3 }& 	\multirow{4}{*}{3, 5} & 	\multirow{4}{*}{[50, 50, 50]$^{T}$}& TIME & 24.348120s & 43.017733s \\			
				& & & Related Error & 0.0442 & 0.0430\\			
					& & & PSNR & 35.1389 & 35.3425\\
					& & & SSIM & 0.9764 & 0.9773\\
						\midrule  \\ 
						\multirow{4}{*}{3000 $\times$ 3000 $\times$ 3}& 	\multirow{4}{*}{5, 5} & 	\multirow{4}{*}{	[50, 50, 50]$^{T}$}& TIME & 5.37095s & 10.489978s \\			
						& & & Related Error & 0.1311 & 0.1262\\			
						& & & PSNR & 26.2742 & 26.5784\\
						& & & SSIM & 0.7358 & 0.7538\\
							\midrule  \\ 
							\multirow{4}{*}{6000 $\times$ 8000 $\times$ 3}& 	\multirow{4}{*}{3, 4} & 	\multirow{4}{*}{	[100, 100, 100]$^{T}$}& TIME & 41.495969s & 110.945373s \\			
							& & & Related Error & 0.0597 & 0.0583\\			
							& & & PSNR & 28.9128& 29.0973\\
							& & & SSIM & 0.9548 & 0.9580\\			
								\midrule  \\ 
								\multirow{4}{*}{4000 $\times$ 6000 $\times$ 3}& 	\multirow{4}{*}{10,  10} & 	\multirow{4}{*}{	[50, 50, 50]$^{T}$}& TIME & 6.891970s & 28.234216s \\			
								& & & Related Error & 0.0396 & 0.0381\\			
								& & & PSNR & 33.5539 & 33.8216\\
								& & & SSIM & 0.8520 & 0.8551\\
									\midrule  \\ 
									\multirow{4}{*}{5304 $\times$ 7952 $\times$ 3}& 	\multirow{4}{*}{4, 4} & 	\multirow{4}{*}{	[100, 100, 100]$^{T}$}& TIME & 39.946396s & 97.926987s \\			
									& & & Related Error & 0.1173 & 0.1108 \\			
									& & & PSNR & 26.0268 & 26.5003\\
									& & & SSIM & 0.8696 & 0.8788\\													 		
					\bottomrule
				\end{tabular}}
				\label{table 5} 
			\end{table}  
	\noindent $\mathbf{Note}$ $\mathbf{3:}$ TSTP-SVD and TT-SVD in Table \ref{table 5} denote truncated STP-SVD and truncated T-SVD, respectively. \\	
    \indent It is obvious from the Table \ref{table 5} that the algorithm introduced in this paper saves a lot of operation time while achieving almost the same compression quality compared with T-SVD.
	\section{Conclusion}\label{sec:6} 
					\indent In this paper, we have introduced a new definition of third-order tensor semi-tensor product. By using this product, we also presented a new type of tensor decomposition strategy and gave the specific algorithm. This decomposition strategy actually generalizes the matrix SVD based on semi-tensor product to third-order tensors. The new decomposition model can achieve data compression to a great extent on the basis of the existing tensor decomposition algorithms. We gave the theoretical analysis and verified it with numerical experiments in our paper.\\
					\indent The advantage of the algorithm which based on semi-tensor product of tensors is that it can achieve data compression by reducing the data storage. The decomposition strategy in this paper will reduce the number of calculations and storage space so that speeding up the operation when we decompose a third-order tensor. While our algorithm certainly has the limitation, that is, we can only approximately decompose the matrix anyway when we decompose a matrix on the basis of semi-tensor product. This will produce errors when we decompose tensors. However, if each frontal slice of target tensor is a low-rank matrix, our algorithm will get a great result. Our decomposition model still has room for improvement. We will investigate how to build a decomposition model based on semi-tensor product for $p$-th order tensors with $p>3$ in future work. And we will consider whether there will be special treatment ideas for tensors with special structures. 





\bibliographystyle{plain} 
					\bibliography{Ref}

\begin{thebibliography}{10}

\bibitem{batselier2017constructive}
K.~Batselier and N.~Wong.
\newblock A constructive arbitrary-degree kronecker product decomposition of
  tensors.
\newblock {\em Numerical Linear Algebra with Applications}, 24(5):e2097, 2017.

\bibitem{kroneckerproduct}
R.~Bellman.
\newblock {\em Introduction to Matrix Analysis, Second Edition}.
\newblock Society for Industrial and Applied Mathematics, 1997.

\bibitem{RePEc:spr:psycho:v:35:y:1970:i:3:p:283-319}
J.D. Carroll and J.J. Chang.
\newblock Analysis of individual differences in multidimensional scaling via an
  n-way generalization of “eckart-young” decomposition.
\newblock {\em Psychometrika}, 35(3):283--319, 1970.

\bibitem{douglas1980candelinc}
J.D. Carroll, P.~Sandra, and J.B. Kruskal.
\newblock Candelinc: A general approach to multidimensional analysis of
  many-way arrays with linear constraints on parameters.
\newblock {\em Psychometrika}, 45(1):3--24, 1980.

\bibitem{cheng2007survey}
D.Z. Cheng, H.S. Qi, and A.C. Xue.
\newblock A survey on semi-tensor product of matrices.
\newblock {\em Journal of Systems Science and Complexity}, 20(2):304--322,
  2007.

\bibitem{cheng2012introduction}
D.Z. Cheng, H.S. Qi, and Y.~Zhao.
\newblock {\em An Introduction to Semi-tensor Product of Matrices and Its
  Applications}.
\newblock World Scientific, 2012.

\bibitem{golub2013matrix}
G.H. Golub and C.F.~Van Loan.
\newblock {\em Matrix Computations}.
\newblock JHU press, 2013.

\bibitem{harshman1970foundations}
R.A. Harshman.
\newblock Foundations of the parafac procedure: Models and conditions for an
  “explanatory” multimodal factor analysis.
\newblock {\em UCLA working papers in phonetics}, 16:1--84, 1970.

\bibitem{harshmannote}
R.A. Harshman.
\newblock Parafac2: Mathematical and technical notes.
\newblock {\em UCLA working papers in phonetics}, 22:30--47, 1972.

\bibitem{harshman1978models}
R.A. Harshman.
\newblock Models for analysis of asymmetrical relationships among n objects or
  stimuli.
\newblock In {\em First Joint Meeting of the Psychometric Society and the
  Society of Mathematical Psychology, Hamilton, Ontario}, 1978.

\bibitem{harshman1996uniqueness}
R.A. Harshman and M.E. Lundy.
\newblock Uniqueness proof for a family of models sharing features of tucker's
  three-mode factor analysis and parafac/candecomp.
\newblock {\em Psychometrika}, 61(1):133--154, 1996.

\bibitem{kiers2000towards}
H.A.L. Kiers.
\newblock Towards a standardized notation and terminology in multiway analysis.
\newblock {\em Journal of Chemometrics: A Journal of the Chemometrics Society},
  14(3):105--122, 2000.

\bibitem{kilmer2013third}
M.E. Kilmer, K.~Braman, N.~Hao, and R.C. Hoover.
\newblock Third-order tensors as operators on matrices: A theoretical and
  computational framework with applications in imaging.
\newblock {\em SIAM Journal on Matrix Analysis and Applications},
  34(1):148--172, 2013.

\bibitem{kilmer2011factorization}
M.E. Kilmer and C.D. Martin.
\newblock Factorization strategies for third-order tensors.
\newblock {\em Linear Algebra and its Applications}, 435(3):641--658, 2011.

\bibitem{t-product}
M.E. Kilmer, C.D. Martin, and L.~Perrone.
\newblock A third-order generalization of the matrix svd as a product of
  third-order tensors.
\newblock {\em Tufts University, Department of Computer Science, Tech. Rep.
  TR-2008-4}, 2008.

\bibitem{kolda2006multilinear}
T.G. Kolda.
\newblock Multilinear operators for higher-order decompositions. {No.
  SAND2006-2081}.
\newblock Technical report, Sandia National Laboratories (SNL), Albuquerque,
  New Mexico, and Livermore, California, 2006.

\bibitem{kolda2009tensor}
T.G. Kolda and B.W. Bader.
\newblock Tensor decompositions and applications.
\newblock {\em SIAM review}, 51(3):455--500, 2009.

\bibitem{liu2022semi}
W.H. Liu, Z.J. Xie, and X.Q. Jin.
\newblock A semi-tensor product of tensors and applications.
\newblock {\em East Asian Journal on Applied Mathematics}, 12(3):696--714,
  2022.

\bibitem{van2000ubiquitous}
C.F.~Van Loan.
\newblock The ubiquitous kronecker product.
\newblock {\em Journal of computational and applied mathematics},
  123(1-2):85--100, 2000.

\bibitem{van1993approximation}
C.F.~Van Loan and N.~Pitsianis.
\newblock Approximation with kronecker products.
\newblock In {\em Linear algebra for large scale and real-time applications},
  pages 293--314. Springer, 1993.

\bibitem{tensorrobust}
C.Y. Lu, J.S. Feng, Y.D. Chen, W.~Liu, Z.C. Lin, and S.C. Yan.
\newblock Tensor robust principal component analysis with a new tensor nuclear
  norm.
\newblock {\em IEEE transactions on pattern analysis and machine intelligence},
  42(4):925--938, 2019.

\bibitem{huynh2008scope}
H.T. Quan and M.~Ghanbari.
\newblock Scope of validity of psnr in image/video quality assessment.
\newblock {\em Electronics letters}, 44(13):800--801, 2008.

\bibitem{tucker1966some}
L.R. Tucker.
\newblock Some mathematical notes on three-mode factor analysis.
\newblock {\em Psychometrika}, 31(3):279--311, 1966.

\bibitem{wang2004image}
Z.~Wang, A.C. Bovik, H.R Sheikh, and E.P. Simoncelli.
\newblock Image quality assessment: from error visibility to structural
  similarity.
\newblock {\em IEEE transactions on image processing}, 13(4):600--612, 2004.

\end{thebibliography}
\end{document}